\documentclass[11pt]{article}
\usepackage[latin1]{inputenc}
\usepackage{amsmath,amsthm,amssymb}
\usepackage{amsfonts}
\usepackage{amsmath,amsthm,amssymb,amscd}
\usepackage{latexsym}
\usepackage{color}
\usepackage{graphicx}
\usepackage{mathrsfs}

\makeatletter \@addtoreset{equation}{section} \makeatother

\allowdisplaybreaks

\numberwithin{equation}{section}
\newtheorem{theorem}{Theorem}[section]
\newtheorem{proposition}[theorem]{Proposition}
\newtheorem{lemma}[theorem]{Lemma}
\newtheorem{remark}[theorem]{Remark}

\newtheorem{definition}[theorem]{Definition}
\theoremstyle{definition}

\begin{document}

\title{Multiplicity and concentration of solutions for a fractional $p$-Kirchhoff type equation}

\author{Wenjing Chen\footnote{
E-mail address:\, {\tt wjchen@swu.edu.cn} (W. Chen), {\tt panhuayu1226@163.com} (H. Pan)}\ \  and Huayu Pan\\
\footnotesize  {School of Mathematics and Statistics, Southwest University,
Chongqing, 400715, P.R. China}}

\date{ }
\maketitle
\begin{abstract}
{This paper is concerned with the following fractional $p$-Kirchhoff equation
\begin{eqnarray*}
\varepsilon ^{sp}M\left( {\varepsilon ^{sp - N}}\iint_{\mathbb{R}^{2N}}\frac{{{{\left| {u(x) - u(y)} \right|}^p}}}{{{{\left| {x - y} \right|}^{N + sp}}}}dxdy\right)(-\Delta)_p^su + V(x){u^{p - 1}} =  {u^{p_s^* - 1}}+f(u),\ \ u>0, \ \mbox{in}\ {\mathbb{R}^N},
\end{eqnarray*}
where $\varepsilon>0$ is a parameter, $M(t)=a+bt^{\theta-1}$ with $a>0$, $b>0$, $\theta>1$, $(-\Delta)_p^s$ denotes the fractional $p$-Laplacian operator with $0<s<1$ and $1<p<\infty$, $N>sp$, $\theta p<p_s^*$ with $p_s^*=\frac{Np}{N-sp}$ is the fractional critical Sobolev exponent, $f$ is a superlinear continuous function with subcritical growth and $V$ is a positive continuous potential. Using penalization method and Ljusternik-Schnirelmann theory, we study the existence, multiplicity and concentration of nontrivial solutions for $\varepsilon>0$ small enough.}


\end{abstract}

\section{Introduction}\label{sec1}

In this article, we are interested in the following fractional $p$-Kirchhoff equation
\begin{equation}\label{1.1}
\varepsilon ^{sp}M\left( {\varepsilon ^{sp - N}}\iint_{\mathbb{R}^{2N}}\frac{{{{\left| {u(x) - u(y)} \right|}^p}}}{{{{\left| {x - y} \right|}^{N + sp}}}}dxdy\right)(-\Delta)_p^su + V(x){u^{p - 1}} =  {u^{p_s^* - 1}}+f(u),\ \ u>0, \ \mbox{in}\ {\mathbb{R}^N},
\end{equation}
where $\varepsilon>0$ is a parameter, $M(t)=a+bt^{\theta-1}$ with $a>0$, $b>0$, $\theta>1$, and $0<s<1$, $1<p<\infty$, $N>sp$, $\theta p<p_s^*$, $p_s^*=\frac{Np}{N-sp}$ is the fractional critical Sobolev exponent, $(-\Delta)_p^s$ denotes the fractional $p$-Laplacian which (up to normalization factors) may be defined for every function $u\in \mathcal{C}_0^\infty(\mathbb{R}^N)$ as
\begin{equation*}
(-\Delta)_p^s=2\mathop {\lim }\limits_{r \to 0} \int_{\mathbb{R}^N\setminus {B_r(x)}}\frac{{|u(x) - u(y){|^{p - 2}}(u(x) - u(y))}}{{|x - y{|^{N + sp}}}}dy,\ \ for \ \ x\in\mathbb{R}^N.
\end{equation*}

When $p=2$, $a=1$ and $b=0$, (\ref{1.1}) is related to the following  fractional Schr\"{o}dinger equation
\begin{equation}\label{S}
\varepsilon^{2s}(-\Delta)^su+V(x)u=h(x,u)\ \ \ \mbox{in}\ \ \mathbb{R}^N,
\end{equation}
which has been proposed by Laskin \cite{L} as a result of expending the Feynman path integral from the Brownian like to the L\'{e}vy like quantum mechanical paths. In the past few years, equation (\ref{S}) has been widely considered, and the multiplicity of solutions to (\ref{S}) has been established by using variational methods under different and suitable assumptions on $V$ and $h$, we refer to \cite{AOO, VA1, VA2, VA3, PAJ, T, S, JMJ} and references therein.\

In the case $s=\varepsilon=1$, $p=2$, $\theta=2$ and $N=3$, equation (\ref{1.1}) becomes the following equation
\begin{equation}\label{SK}
-\Big(a+b\int_{\mathbb{R}^3}|\nabla u|^2dx\Big)\Delta u+V(x)u=h(x,u)\ \ \ \mbox{in}\ \ \mathbb{R}^3,
\end{equation}
which is related to the stationary analogue of the equation
\begin{equation}\label{WSK}
\rho u_{tt}-\left(\frac{\rho_0}{h}+\frac{E}{2L}\int_0^L|u_x|^2dx\right)u_{xx}=0,
\end{equation}
introduced by Kirchhoff \cite{K} as an extension of the classical D'Alembert's wave equation for describing the transversal oscillations of a stretched string. The Kirchhoff's model takes into account the charges in length of the string produced by transverse vibrations. Here $u=u(x,t)$ is the transverse string displacement at the space coordinate $x$ and time $t$, $L$ is the length of the string, $h$ is the area of the cross section, E is Young's modulus of the material, $\rho$ is the mass density, and $P_0$ is the initial tension; see \cite{SB,SIP}. In the classical framework, He and Zou \cite{HZ} proved the multiplicity and concentration behavior of solutions for small $\varepsilon>0$ to the following perturbed Kirchhoff equation.
\begin{equation}\label{SKi}
-\Big(a\varepsilon^2+b\varepsilon\int_{\mathbb{R}^3}|\nabla u|^2dx\Big)\Delta u+V(x)u=g(u)\ \ \ \mbox{in}\ \ \mathbb{R}^3,
\end{equation}
under the following global condition on $V$ due to Rabinowitz \cite{PR}
\begin{equation*}
\mathop {\inf }\limits_{x \in \mathbb{R}^3} V(x)<\mathop {\lim \inf }\limits_{|x| \to \infty } V(x):=V_\infty\leq \infty,
\end{equation*}
and $g$ is a subcritical nonlinearity. Subsequently, Wang et al. \cite{WTXZ} investigated the multiplicity and concentration phenomenon for (\ref{SKi}) in the presence of the critical term. Under local conditions on the potential $V$, by the generalized Nehari manifold method and Ljusternik-Schnirelmann theory, Figueiredo and Junior \cite{GJ} obtained the multiplicity result for a subcritical Kirchhoff equation. The existence and concentration of positive solutions for (\ref{SKi}) with critical growth has been considered in \cite{HLP}.

Moreover, a great attention has been focused on the study of the fractional $p$-Laplacian problem.
In particular, Pucci et al. \cite{PXZ} obtained the multiplicity result for nonhomogeneous fractional Kirchhoff-Schr\"{o}dinger equation.
Assuming that $V$ satisfies a Bartsch-Wang type condition, Fiscella and Pucci \cite{FP} considered stationary fractional Kirchhoff $p$-Laplacian equations involving a Hardy potential and different critical nonlinearities. The existence result for a class of quasilinear Kirchhoff system involving the $p$-Laplacian was proved by Xiang et al. \cite{MRZ}. Recently, Ambrosio and Isernia \cite{VA4} proved the multiplicity result and the number of positive solutions is related with the topology of the set, where the potential $V$ attains its minimum. Then, by using some appropriate variational arguments, they obtained the existence result when $f(u)=u^{q-1}+\gamma u^{r-1}$ in (\ref{1.1}) with $N=3$, where $\gamma>0$ is sufficiently small, and the powers $q$ and $r$ satisfy $2p<q<p_s^*\leq r$.

Motivated by the above works and the fact that several interesting questions arise from the search of nontrivial non negative solutions, our purpose is to study the existence, multiplicity and concentration of positive solutions to (\ref{1.1}) under suitable  assumptions on the potential $V$ and the nonlinearity $f$.

In order to state our main result, we will assume that $V: \mathbb{R}^N\rightarrow \mathbb{R}$ satisfies the following assumptions introduced by del Pino and Felmer \cite{PF}:

$(V_1)$ $V \in{\mathcal{C}}({\mathbb{R}^N},\mathbb{R})$ and $V_1:=\mathop {\inf  }\limits_{x\in {\mathbb{R}^N} } V(x) > 0$;\

$(V_2)$ there is a bounded open set $\Lambda \subset \mathbb{R}^N$ such that
\begin{equation*}
\mathop {\min }\limits_{\partial \Lambda }V(x)>\mathop {\inf }\limits_{\Lambda }V=:V_0>0,
\end{equation*}
and the set $M: = \left\{ {x \in \Lambda :V(x) = {V_0}} \right\} \ne \emptyset $.\

Moreover, we give some assumptions on $f$ as follows. Since we are looking for positive solutions, we may assume that $f(t)=0$ for $t<0$. Let $f\in{\mathcal{C}^1}({\mathbb{R}^+},{\mathbb{R}})$  be such that

$(F_1)$ $f(t)=o(t^{\theta p-1})$, as $t\rightarrow 0^+$;

$(F_2)$ there exist $q, \sigma\in(\theta p,p_s^*)$ and $\lambda>0$ such that
\begin{equation*}
f(t)\geq \lambda t^{q-1},\ \forall t>0,\ \mbox{and}\ \mathop {\lim }\limits_{t\rightarrow \infty }\frac{f(t)}{t^{\sigma-1}}=0,
\end{equation*}
where $\lambda$ satisfies
\begin{itemize}
  \item $\lambda>0$ if either $N\geq sp^2$, or $sp<N<sp^2$ and $p_s^*-\frac{p}{p-1}<q<p_s^*$,
  \item $\lambda$ is sufficiently large if $sp<N<sp^2$ and $\theta p<q\leq p_s^*-\frac{p}{p-1}$;
\end{itemize}

$(F_3)$ there exists $\mu\in(\theta p,p_s^*)$ such that $0<\mu F(t)\leq f(t)t$ for all $t\geq0$;

$(F_4)$ the map $t\mapsto \frac{f(t)}{t^{\theta p-1}}$ is increasing in $(0,\infty)$.

\smallskip

To describe our main result, we recall the following definition of Ljusternik-Schnirelmann category.
\begin{definition}
If $Y$ is a given closed set of a topological space $X$, we denote by $cat_YX$ the Ljusternik-Schnirelmann category of $Y$ in $X$, that is the least number of closed and contractible sets in $X$ which cover $Y$.
\end{definition}

For more properties about the Ljusternik-Schnirelmann category, we refer to \cite{BC,W}.

The main result can be stated as follows.
\begin{theorem}\label{mainthx}
Suppose that $(V_1)$-$(V_2)$ and $(F_1)$-$(F_4)$ hold. Then, for any $\delta>0$ such that
\begin{equation*}
{M_\delta } = \left\{ {x \in \mathbb{R}^N:dist\left( {x,M} \right) \le \delta } \right\} \subset \Lambda,
\end{equation*}
there exists ${\varepsilon _\delta }>0$ such that, for any $\varepsilon  \in \left( {0,{\varepsilon _\delta }} \right)$, problem (\ref{1.1}) has at least $ca{t_{{M_\delta }}}(M)$ positive solutions. Moreover, if $u_\varepsilon$ denotes one of these solutions and $x_\varepsilon \in \mathbb{R}^N$ is a global maximum point of ${u_\varepsilon }$, then we have
\begin{equation*}
\mathop {\lim }\limits_{\varepsilon  \to 0} V({x_\varepsilon }) = {V_0}
\end{equation*}
and there exists $C>0$ such that
\begin{equation*}
0 < {u_\varepsilon (x)} \le \frac{{C{\varepsilon ^{N + sp}}}}{{{\varepsilon ^{N + sp}} + {{\left| {x - {x_\varepsilon }} \right|}^{N + sp}}}},\ \ \forall x\in \mathbb{R}^N.
\end{equation*}
\end{theorem}
\begin{remark}
We obtain the result of Theorem \ref{mainthx} for $1<\theta<\frac{N}{N-sp}$ and all $a, b>0$. This expands results of \cite{CG} with $\alpha=0$: $1<\theta<\frac{N}{N-sp}$ and $a=0$, $b>0$. Moreover, we discuss both the case $sp<N<sp^2$ and $\theta p<q\leq p_s^*-\frac{p}{p-1}$, which is more complicated than that of \cite{CG,CSS}.
\end{remark}
\begin{remark}
We give an example of function $f$ satisfying conditions $(F_1)$-$(F_4)$, $f(t)=\sum\limits_{i=1}^{k} \alpha_i t^{q_i-1}$ with $\alpha_i\geq 0$ not all null and $\theta p<q_i<p_s^*$ for all $i\in \{1,\ldots,k\}$.
\end{remark}

Under the local assumption $(V_2)$, Theorem \ref{mainthx} will be proved by the penalization methods developed by del Pino and Felmer \cite{PF}. Comparing our result with the result of the classical Kirchhoff problem (\ref{SKi}), there are some new difficulties:

(i) The key point is to overcome the compactness for the associated Lagrange-Euler functional, namely, the Palais-Smale ((PS) for short) condition. Since the nonlinearity term in  equation (\ref{1.1}) contains the critical Sobolev term, the functional does not satisfy the Palais-Smale condition in all range, we will use a fractional version of the concentration compactness principle to show that the energy
functional satisfies the local $(PS)_c$ condition for $c$ less than a suitable compactness threshold when the perturbation term $f$ satisfies condition $(F_2)$.

On the other hand, the compactness threshold is usually related to the best fractional critical Sobolev constant $S$, namely
\begin{eqnarray}\label{criticalfrac}
S:=\inf\limits_{u \in D^{s,p}(\mathbb{R}^N)\setminus\{0\}} \frac{\iint_{\mathbb{R}^{2N}}\frac{{{{\left| {u(x) - u(y)} \right|}^p}}}{{{{\left| {x - y} \right|}^{N + sp}}}}dxdy}{\left(\int_{\mathbb{R}^N}|u|^{p_s^\ast}\right)^{p/p_s^\ast}}.
\end{eqnarray}
For the critical fractional case with $p\neq 2$, the main difficulty is the lack of an explicit formula for minimizers of $S$ which is very often a key tool to handle the estimates leading to the compactness range of the functional. It was conjectured that, up to a multiplicative constant,  all minimizers are of the form $U((x-x_0)/\varepsilon),$ with
\begin{equation*}
U(x)=(1+|x|^{\frac{p}{p-1}})^{-\frac{N-ps}{p}},\quad x\in \mathbb{R}^N.
\end{equation*}
This conjecture was proved in \cite{classif} for $p=2$, but for $p\neq 2$, it is not even
known if these functions are minimizers of $S$. As in \cite{CG,CSS}, we can overcome this difficulty by the optimal asymptotic behavior of minimizers, which was obtained in \cite{LS}.

(ii) The Lagrange-Euler functional of problem (\ref{1.1})  is related to the presence of the Kirchhoff term $[u]_{s,p}^p(-\Delta)_p^s$ and the lack of compactness caused by the unboundedness of the domain.

(iii) The non-Hilbert structure of the fractional Sobolev space $W^{s,p}(\mathbb{R}^N)$ when $p\neq2$ makes our study rather tough, we have to develop some suitable arguments which take care of the nonlocal character of the leading operator $(-\Delta)_p^s$, see Lemmas \ref{conv1} and \ref{conv2}. Moreover, we borrow ideas of Moser iteration argument \cite{Moser} with H\"{o}lder continuity result established for $(-\Delta)_p^s$, which implies that the solution of the modified problem is also a solution of the original one.

The paper is organized as follows. In Section 2, we first present some preliminaries, and we introduce the modified problem in Section \ref{modip}. 
In Section 3, we will prove the limiting problem has a positive ground state solution. In Section 4, we show the modified problem has multiple solutions. Section 5 is devoted to prove Theorem \ref{mainthx}.

\section{Preliminaries}
In this section, we outline the variational framework for studying problem (\ref{1.1}) and list some preliminary lemmas which will be used later. In the sequel, we denote by $\|\cdot\|_p$ the usual norm of the space $L^{p}(\mathbb{R}^N)$ and denote by $|\cdot|_\infty$ the usual norm of the space $L^{\infty}(\mathbb{R}^N)$, the letters $C_i(i=1,2,\ldots)$ and $C$ denote some positive constants. We denote by $B_r(x)$ the ball centered at $x\in \mathbb{R}^N$ with radius $r>0$. For $s\in(0,1)$, we define $D^{s,p}(\mathbb{R}^N)$ as the closure of $\mathcal{C}_c^\infty(\mathbb{R}^N)$ with respect to
\begin{equation*}
\left[ u \right]_{s,p}^p: = \iint_{\mathbb{R}^{2N}}\frac{{{{\left| {u(x) - u(y)} \right|}^p}}}{{{{\left| {x - y} \right|}^{N + sp}}}}dxdy.
\end{equation*}
Consider the fractional Sobolev space
\begin{equation*}
{W^{s,p}}(\mathbb{R}^N) = \left\{ {u \in {L^p}(\mathbb{R}^N):{{\left[ u \right]}_{s,p}} <  + \infty } \right\}.
\end{equation*}
Then ${W^{s,p}}(\mathbb{R}^N)$ is a Banach Space with respect to the norm
\begin{equation*}
\|u\|_{s,p}^p:={\left[ u \right]}_{s,p}^p+\|u\|_p^p.
\end{equation*}\

\begin{lemma}\label{SE}\cite{NEGE}
Let $s\in(0,1)$ and $p\in(1,\infty)$ be such that $sp<N$. Then there exists a constant $C_*:=C_*(s,p)>0$ such that, for any $u\in {D^{s,p}}(\mathbb{R}^N)$, we have
\begin{equation*}
\left|u\right|_{p_s^*}^p\leq C_*{\left[ u \right]}_{s,p}^p.
\end{equation*}
Moreover, $ {W^{s,p}}(\mathbb{R}^N)$ is continuously embedded in $L^q(\mathbb{R}^N)$ for any $q\in \left[p,p_s^*\right]$, and compactly in $L_{loc}^q(\mathbb{R}^N)$ for any $q\in \left[1,p_s^*\right)$.
\end{lemma}
\begin{lemma}\label{Lion}\cite{AI}
Let $s\in(0,1)$ and $p\in(1,\infty)$ be such that $sp<N$, and $r\in[p,p_s^*)$. If $\{u_n\}_{n\in \mathbb{N}}$ is a bounded sequence in $W^{s,p}(\mathbb{R}^N)$ and
\begin{equation*}
\mathop {\lim }\limits_{n \to \infty } \mathop {\sup }\limits_{y \in \mathbb{R}^N} \int_{{B_R}(y)} {|{u_n}{|^r}dx}  = 0,
\end{equation*}
where $R>0$, then $u_n\rightarrow 0$ in $L^\sigma(\mathbb{R}^N)$ for all $\sigma\in \left(p,p_s^*\right)$.
\end{lemma}
\begin{lemma}\label{Lion1}\cite{AI}
Let $u\in W^{s,p}(\mathbb{R}^N)$ and $\phi\in\mathcal{C}_c^\infty(\mathbb{R}^N)$ be such that $0\leq\phi\leq 1$, $\phi=1$ in $B_1$ and $\phi=0$ in $B_2^c$. Set $\phi_\varsigma(x)=\phi(x/\varsigma)$. Then,
\begin{equation*}
{\left[ {u{\phi _\varsigma} - u} \right]_{s,p}} \to 0\ \  \text{and}\ \   |u\phi _\varsigma - u|_p \to 0,
\end{equation*}
as $\varsigma\to0$.
\end{lemma}
Next, we introduce the well-known Simon inequality.
\begin{lemma}\label{Simon}\cite{Si}
There exist constants $c_p, C_p>0$ such that for any $a,b\in {\mathbb{R}}$, it holds
\begin{eqnarray*}
|a-b|^p \leq\left\{ \arraycolsep=1.5pt
   \begin{array}{ll}
{c_p}\left( {|a{|^{p - 2}}a - |b{|^{p - 2}}b} \right)\left(a - b\right)\ \ & \mbox{if}\ p\geq2\\
{C_p}{\left[ {\left( {|a{|^{p - 2}}a - |b{|^{p - 2}}b} \right)(a - b)} \right]^{\frac{p}{2}}}{\left( {|a{|^p} + |b{|^p}} \right)^{\frac{{(2 - p)}}{2}}}\ \ & \mbox{if}\ 1<p<2.
\end{array} \right.
\end{eqnarray*}
\end{lemma}

\smallskip

In what follows, we provide some useful estimates, which will be needed to overcome the difficulty coming from the critical nonlinear term. Let $S$ be the best fractional Sobolev embedding constant defined by (\ref{criticalfrac}).
As shown in \cite{LS}, there exists a radially symmetric nonnegative decreasing minimizer $U=U(r)$ for $S$ such that $( - \Delta )_p^sU = {U^{p_s^* - 1}}$ in $\mathbb{R}^N$, and $[U]_{s,p}^p = |U|_{p_s^*}^{p_s^*} = S^{\frac{N}{{sp}}}$. Moreover, $U\in L^\infty(\mathbb{R}^N)\cap\mathcal{C}^0(\mathbb{R}^N)$,
\begin{equation*}
\mathop {\lim }\limits_{|x| \to \infty } |x{|^{\frac{{N - sp}}{{p - 1}}}}U(x) = {U_\infty } \in \mathbb{R}\setminus\{0\}
\end{equation*}
and verifies the following decay estimates.
\begin{lemma}\label{c1}\cite{LS}
There exist constants $c_1, c_2>0$ and $\kappa>1$ such that for all $r\geq1$,
\begin{equation}\label{2.1}
\frac{{{c_1}}}{{{r^{\frac{{N - sp}}{{p - 1}}}}}} \le U(r) \le \frac{{{c_2}}}{{{r^{\frac{{N - sp}}{{p - 1}}}}}}
\end{equation}
and
\begin{equation*}
\frac{{U(\kappa r)}}{{U(r)}} \le \frac{1}{2}.
\end{equation*}
\end{lemma}
For any $\epsilon>0$, we consider the following family of minimizers for $S$ given by
\begin{equation*}
{U_\epsilon }(x): = \frac{1}{{{\epsilon^{\frac{{N - sp}}{p}}}}}U\left(\frac{{|x|}}{\epsilon}\right).
\end{equation*}
For $\epsilon, \delta>0$, set
\begin{equation*}
{m_{\epsilon ,\delta }}: = \frac{{{U_\epsilon}(\delta )}}{{{U_\epsilon}(\delta ) - {U_\epsilon}(\kappa \delta )}},
\end{equation*}
and define
\begin{eqnarray*}
{g_{\epsilon,\delta }}(t): =\left\{ \arraycolsep=1.5pt
   \begin{array}{ll}
0\ \ & \text{if}\ {0\leq t\leq {U_\epsilon}(\kappa \delta )},\\
m_{\epsilon,\delta }^p(t - {U_\epsilon}(\kappa \delta )\ \ & \text{if}\ {{U_\epsilon}(\kappa \delta )\leq t\leq{U_\epsilon}( \delta )},\\
t + {U_\epsilon}(\delta )(m_{\epsilon,\delta }^{p - 1} - 1)\ \ & \text{if}\ {t\geq {U_\epsilon}( \delta )},
\end{array} \right.
\end{eqnarray*}
and
\begin{eqnarray*}
{G_{\epsilon,\delta }}(t): = \int_0^t {{{({{g'}_{\epsilon,\delta }}(\tau ))}^{\frac{1}{p}}}} d\tau  = \left\{ \arraycolsep=1.5pt
   \begin{array}{ll}
0\ \ & \text{if}\ {0\leq t\leq {U_\epsilon}(\kappa \delta )},\\
{m_{\epsilon,\delta }}(t - {U_\epsilon}(\kappa \delta ))\ \ & \text{if}\ {{U_\epsilon}(\kappa \delta )\leq t\leq{U_\epsilon}( \delta )},\\
t \ \ & \text{if}\ {t\geq {U_\epsilon}( \delta )}.
\end{array} \right.
\end{eqnarray*}
Let us observe that ${g_{\epsilon,\delta }}$ and ${G_{\epsilon,\delta }}$ are nondecreasing and absolutely continuous functions. Now, we consider the radially symmetric nonincreasing function
\begin{equation}\label{ue}
{u_{\epsilon,\delta }}(r)={G_{\epsilon,\delta }}({U_{\epsilon}}(r)),
\end{equation}
which, in view of the definition of ${G_{\epsilon,\delta }}$, satisfies
\begin{eqnarray*}
{u_{\epsilon,\delta }}(r) =\left\{ \arraycolsep=1.5pt
   \begin{array}{ll}
{U_{\epsilon}}(r)\ \ & \text{if}\ {r\leq\delta}\\
0\ \ & \text{if}\ {r\geq\kappa\delta}.
\end{array} \right.
\end{eqnarray*}

\begin{lemma}\label{a}\cite{SKS}
There exists a constant $C=C(N,p,s)>0$ such that for any $\epsilon<\frac{\delta}{2}$ the following estimates hold
\begin{equation*}
[{u_{\varepsilon ,\delta }}]_{s,p}^p \le S^{\frac{N}{{sp}}} + C\left({\left(\frac{\epsilon}{\delta }\right)^{\frac{{N - sp}}{{p - 1}}}}\right),
\end{equation*}
\begin{equation*}
\|{u_{\varepsilon ,\delta }}\|_{p_s^*}^{p_s^*} \ge S^{\frac{N}{{sp}}} - C\left({\left(\frac{\epsilon}{\delta }\right)^{\frac{N}{{p - 1}}}}\right).
\end{equation*}
\end{lemma}
\begin{lemma}\label{b}\cite{AFI,CSS}
There exists a constant $C=C(N,p,s)>0$ such that for any $\epsilon<\frac{\delta}{2}$
\begin{eqnarray*}
\|{u_{\epsilon,\delta }}\|_p^p \le \left\{ \arraycolsep=1.5pt
   \begin{array}{ll}
C\epsilon^{sp}\ \ & \text{if}\ {N>sp^2},\\
C\epsilon^{sp}\log\left(\frac{\delta }{\epsilon}\right)\ \ & \text{if}\ {N=sp^2},\\
C{\epsilon^{\frac{{N - sp}}{{p - 1}}}}{\delta ^{\frac{{s{p^2} - N}}{{p - 1}}}} - C{\epsilon^{sp}}\ \ & \text{if}\ {N<sp^2},
\end{array} \right.
\end{eqnarray*}
and
\begin{eqnarray*}
\|{u_{\epsilon,\delta }}\|_q^q \le \left\{ \arraycolsep=1.5pt
   \begin{array}{ll}
C{\epsilon^{N - \frac{{(N - sp)}}{p}q}}\ \ & \text{if}\ {q > \frac{{N(p - 1)}}{{N - sp}}},\\
C{\epsilon^{N - \frac{{(N - sp)}}{p}q}}|\log\epsilon|\ \ & \text{if}\ {q = \frac{{N(p - 1)}}{{N - sp}}},\\
C{\epsilon^{\frac{{(N - sp)q}}{{p(p - 1)}}}}\ \ & \text{if}\ {q <\frac{{N(p - 1)}}{{N - sp}}}.
\end{array} \right.
\end{eqnarray*}
\end{lemma}

\section{The modified problem}\label{modip}

Taking the change of variable $x\mapsto\varepsilon x$, equation (\ref{1.1}) becomes the following equation
\begin{eqnarray}\label{1.1*}
(a + b\left[ u \right]_{s,p}^{(\theta-1)p})(-\Delta)_p^su + V(\varepsilon x){u^{p - 1}} ={u^{p_s^* - 1}} + f(u), \  u>0,\  \mbox{in}\ {\mathbb{R}^N}.
\end{eqnarray}
Now, we introduce a penalized function in the spirit of \cite{PF}, which will be fundamental to obtain our main result.
Let $K>2$ and $a>0$ be such that $a^{p_s^*-1}+f(a)=\frac{V_0}{K} a^{p-1}$, and we consider the function
\begin{eqnarray*}
\tilde f(t): = \left\{ \arraycolsep=1.5pt
   \begin{array}{ll}
{({t^ + })^{p_s^* - 1}}+ f(t)\ \ &\mbox{if}\ t\leq a,\\
\frac{{{V_0}}}{K}{t^{p - 1}}\ \ &\mbox{if}\ t>a.
\end{array} \right.
\end{eqnarray*}

If $\chi_\Lambda$ denotes the characteristic function of the set $\Lambda$, we introduce the penalized nonlinearity $g: \mathbb{R}^N\times\mathbb{R}\rightarrow\mathbb{R}$ by setting
\begin{equation}\label{g}
g(x,t)=\chi_\Lambda(x)\left(t^{p_s^*-1}+f(t)\right)+\left(1-\chi_\Lambda(x)\right)\tilde f(t).
\end{equation}
From conditions $(F_1)$-$(F_4)$, it is easy to check that $g$ satisfies the following properties.\

$(G_1)$ $\mathop {\lim }\limits_{t \to {0^ + }} \frac{{g(x,t)}}{{{t^{\theta p-1}}}} = 0$ uniformly in $x\in \mathbb{R}^N$;

$(G_2)$ $g(x,t)\leq t^{p_s^*-1}+ f(t)$ for all $x\in \mathbb{R}^N, t>0$;\

$(G_3)$ $(i)$ $0<\mu G(x,t):=\mu\int_0^t {g(x,\tau )} d\tau \leq g(x,t)t$ for all $x\in \Lambda$ and $t>0$,\

\ \ \ \ \ \ \ $(ii)$ $0\leq pG(x,t)<g(x,t)t\leq \frac{V_1}{K} a^{p}$ for all $x\in \mathbb{R}^N\setminus\Lambda$ and $t>0$;

$(G_4)$ the map $t\mapsto \frac{g(x,t)}{t^{\theta p-1}}$ is increasing for all $x\in\Lambda$ and $t>0$, and the map $t\mapsto \frac{g(x,t)}{t^{\theta p-1}}$ is increasing for all $x\in\mathbb{R}^N\setminus\Lambda$ and $t\in(0,a)$.\\
Then, we consider the following modified problem
\begin{eqnarray}\label{1.2}
(a + b\left[ u \right]_{s,p}^{(\theta-1)p})(-\Delta)_p^su  + V(\varepsilon x){u^{p - 1}} = g(\varepsilon x,u),\ \ u>0,\ \   \mbox{in}\ {\mathbb{R}^N}.
\end{eqnarray}
Let us note that if $u$ is a solution of (\ref{1.1}) such that
\begin{equation*}
|u(x)| \le a\ \ \text{for}\ \text{any}\ x\in \mathbb{R}^N\setminus\Lambda_\varepsilon,
\end{equation*}
where $\Lambda_\varepsilon:=\left\{x\in \mathbb{R}^N: \varepsilon x\in\Lambda\right\}$, then $u$ is also a solution of (\ref{1.1*}). Therefore, in order to study weak solutions of (\ref{1.2}), we seek the critical points of the following functional
\begin{equation*}
{J_\varepsilon }(u) = \frac{1}{p}\|u\|_\varepsilon^p + \frac{b}{{\theta p}}[u]_{_{s,p}}^{\theta p}  - \int_{\mathbb{R}^N} {G(\varepsilon x,u)} dx,
\end{equation*}
which is well-defined for all $u: \mathbb{R}^N\rightarrow\mathbb{R}$ belonging to the following fractional space
\begin{equation*}
{\mathcal{W}}_\varepsilon:=\left\{u\in{W^{s,p}}(\mathbb{R}^N):\int_{\mathbb{R}^N} {V(\varepsilon x)} |u{|^p}dx<+\infty\right\}
\end{equation*}
endowed with the norm
\begin{equation*}
\|u\|_\varepsilon ^p: = a[u]_{s,p}^p +\int_{\mathbb{R}^N} {V(\varepsilon x)} |u{|^p}dx.
\end{equation*}
Standard arguments show that ${J_\varepsilon }\in\mathcal{C}^1({\mathcal{W}}_\varepsilon,\mathbb{R})$ and its differential is given by
\begin{align*}
\left\langle {{J'_\varepsilon }(u),v} \right\rangle =&(a + b[u]_{s,p}^{(\theta-1)p}) \iint\limits_{\mathbb{R}^{2N}}\frac{{|u(x) - u(y){|^{p - 2}}(u(x) - u(y))(v(x) - v(y))}}{{|x - y{|^{N + sp}}}}dxdy\\
&+\int_{\mathbb{R}^N} {V(\varepsilon x)} |u{|^{p-2}}uvdx-\int_{\mathbb{R}^N} g(\varepsilon x,u)vdx
\end{align*}
for any $u,v\in {\mathcal{W}}_\varepsilon$. Let us introduce the Nehari manifold associated with (\ref{1.2}), that is,
\begin{equation*}
{\mathcal{N}}_\varepsilon:=\left\{u\in{{\mathcal{W}}_\varepsilon\setminus\{0\}}:\left\langle {{{J'}_\varepsilon }(u),u} \right\rangle =0\right\}.
\end{equation*}
We first shown that ${J_\varepsilon }$ possesses geometric structure.
\begin{lemma}\label{mpt}
The functional ${J_\varepsilon }$ satisfies the following conditions:\

$(i)$ there exist $\alpha, \rho>0$ such that ${J_\varepsilon }(u)\geq\alpha$ for any $u\in{\mathcal{W}}_\varepsilon$ such that $\|u\|_\varepsilon=\rho$;\

$(ii)$ there exists $e\in{\mathcal{W}}_\varepsilon$ with $\|e\|_\varepsilon>\rho$ such that ${J_\varepsilon }(e)<0$.
\end{lemma}
\begin{proof}
$(i)$ For $\varepsilon>0$ small, by $(G_1)$ and $(G_2)$, we have
\begin{equation}\label{g1}
|g(\varepsilon x,t)|\leq \varepsilon |t|^{\theta p-1}+C_\varepsilon |t|^{p_s^*-1},\ \ \text{for}\ \text{all}\ \ x\in \mathbb{R}^N,\ \ t\in \mathbb{R}.
\end{equation}
Therefore, for any $u\in\mathcal{W}_\varepsilon\setminus\{0\}$, we get
\begin{equation*}
{J_\varepsilon(u) }\geq\frac{{\|u\|_\varepsilon ^p}}{p} - \frac{{{C_1}\varepsilon }}{{\theta p}}\|u\|_\varepsilon ^{\theta p} - \frac{{{C_2}{C_\varepsilon }}}{{p_s^*}}\|u\|_\varepsilon ^{p_s^*}.
\end{equation*}
Hence, since $\theta>1$, there exist $\alpha, \rho>0$ such that ${J_\varepsilon }(u)\geq\alpha$ with $\|u\|_\varepsilon=\rho$.\

$(ii)$ For all $t>0$, we infer that
\begin{align}
{J_\varepsilon }(u) \leq& \frac{t^p}{p}\|u\|_\varepsilon^p + \frac{bt^{\theta p}}{{\theta p}}[u]_{_{s,p}}^{\theta p}  - \int_{\Lambda_\varepsilon} {G(\varepsilon x,tu)} dx\nonumber\\
\leq& \frac{t^p}{p}\|u\|_\varepsilon^p + \frac{bt^{\theta p}}{{\theta p}}[u]_{_{s,p}}^{\theta p}  - \frac{t^{p_s^*}}{p_s^*} \int_{\Lambda_\varepsilon} {|u^+|^{p_s^*}} dx.
\end{align}
Thanks to $\theta p<p_s^*$, the conclusion $(ii)$ holds.
\end{proof}

\begin{lemma}\label{kmax}
Define the function $H(t)=\frac{a}{p}{S}{t^p} + \frac{b}{{\theta p}}S^\theta{t^{\theta p}} - \frac{1}{{p_s^*}}{t^{p_s^*}}$. Then there exists a unique ${\hat T}>0$ such that
\[
H(\hat T)=\sup\limits_{t\geq0}H(t).
\]
\end{lemma}

\begin{proof}
For $t>0$, we have
\begin{equation}\label{dK}
H'(t) = {t^{p - 1}}\hat H(t),\ \ \hat H(t) = a{S} + bS^\theta{t^{(\theta-1)p}} - {t^{p_s^* - p}}.
\end{equation}
Moreover,
\begin{equation*}
\hat H'(t) = {t^{(\theta-1)p - 1}}\left((\theta-1)pbS^\theta - (p_s^* - p){t^{p_s^* - \theta p}}\right).
\end{equation*}
Thus there exists a unique $\hat T>0$ such that $\hat H>0$ if $t\in(0,\hat T)$ and $\hat H<0$ if $t\in(\hat T,+\infty)$. Thus, $\hat T$ is the unique maximum point of $H(t)$.
\end{proof}

In view of Lemma \ref{mpt}, we can define the minimax level
\begin{equation*}
{c_\varepsilon } = \mathop {\inf }\limits_{\gamma  \in {\Gamma _\varepsilon }} \mathop {\max }\limits_{t \in [0,1]} {J_\varepsilon }(\gamma (t)),
\end{equation*}
where ${\Gamma _\varepsilon } = \left\{ {\gamma  \in C([0,1]),{\mathcal{W}_\varepsilon }):\gamma (0) = 0\ \ \text{and}\ \ {J_\varepsilon }(\gamma (1)) < 0} \right\}$. From \cite[Theorem 4.2]{W}, it is standard to verify that $c_\varepsilon$ can be characterized as follows
\begin{equation}\label{c_e}
{c_\varepsilon } = \mathop {\inf }\limits_{u \in {\mathcal{W}_\varepsilon }\setminus\{0\}} \mathop {\sup }\limits_{t \ge 0} {J_\varepsilon }(tu) = \mathop {\inf }\limits_{u \in \mathcal{N}_\varepsilon} {J_\varepsilon }(u).
\end{equation}

\begin{lemma}\label{compare}
There holds
\begin{equation*}
{c_\varepsilon } <{c_0}= : H(\hat T)
\end{equation*}
for all $\varepsilon>0$, where $H(\hat{T})$ was given in Lemma \ref{kmax}.
\end{lemma}
\begin{proof}
Let $u_{\epsilon,\delta}$ be the function defined in (\ref{ue}) such that $supp(u_{\epsilon,\delta})\subset B_{\theta\delta}\subset\Lambda_\varepsilon$. For simplicity, we take $\delta=1$ and we set $u_\epsilon:=u_{\epsilon,1}$. Then, by $(F_2)$, we have
\begin{align*}
{J_\varepsilon }(tu_\epsilon) \leq \frac{at^{p}}{{p}}[u_\epsilon]_{_{s,p}}^{p}+\frac{V_2t^{p}}{{p}}\|u_\epsilon\|_p^{p} + \frac{bt^{\theta p}}{{\theta p}}[u_\epsilon]_{_{s,p}}^{\theta p}  - \frac{\lambda t^q}{q}\|u_\epsilon\|_q^{q} - \frac{ t^{p_{s^*}}}{{p_{s^*}}}\|u_\epsilon\|_{p_{s^*}}^{{p_{s^*}}}\rightarrow-\infty\ \ \mbox{as}\ \ t\rightarrow\infty,
\end{align*}
where $V_2=|V|_\infty$. Therefore, there exists $t_\epsilon>0$ such that
\begin{equation*}
{J_\varepsilon }(t_\epsilon{u_\epsilon})=\mathop {\max }\limits_{t\geq 0} {J_\varepsilon }(t{u_\epsilon}).
\end{equation*}
Next, we can show that there exist $t_1, t_2>0$ (independent of $\epsilon>0$) satisfying $t_1\leq t_\epsilon\leq t_2$. Indeed, since $\left\langle {{J'_\varepsilon }({t_\epsilon }{u_\epsilon }),{u_\epsilon }} \right\rangle  = 0$, we have
\begin{equation*}
t_\epsilon^p{\|u_\epsilon\|_\epsilon^p}+bt_\epsilon^{\theta p}[u_\epsilon]_{_{s,p}}^{\theta p}=t_\epsilon^{p_s^*}{\|u_\epsilon\|_{p_s^*}^{p_s^*}}+\int_{\mathbb{R}^N} f(t_\epsilon {u_\epsilon}){t_\epsilon u_\epsilon} dx.
\end{equation*}
Moreover
\begin{equation*}
t_\epsilon^p{\|u_\epsilon\|_\epsilon^p}+bt_{\epsilon}^{\theta p}[u_\epsilon]_{_{s,p}}^{\theta p}\geq {t_\epsilon^{p_s^*}}{\|u\|_{p_s^*}^{p_s^*}},
\end{equation*}
which implies that $|t_\epsilon|\leq t_2$. On the other hand, there exists $t_1>0$ such that $t_\varepsilon>t_1$ for $\epsilon>0$ small. Otherwise, there exists $\epsilon_n\to 0$ as $n\to \infty$ such that $t_{\epsilon_n}\to 0$ as $n\to \infty$. Thus,
\begin{equation*}
0<c_\varepsilon\leq\mathop {\max }\limits_{t\geq 0} {J_\varepsilon }(t{u_\epsilon})\to 0,
\end{equation*}
which gives a contradiction, we can see that
\begin{align*}
{J_\varepsilon }(t_\epsilon u_\epsilon) \leq& \frac{at_\epsilon^p}{p}[u_\epsilon]_{s,p}^{p}  + \frac{b t_\epsilon^{\theta p}}{{\theta p}}[u_\epsilon]_{s,p}^{\theta p}  -  \frac{t_\epsilon^{p_s^*}}{p_s^*}\|u_\epsilon\|_{{p_s^*}}^{p_s^*}+\frac{V_2 t_\epsilon^p}{p}\|u_\epsilon\|_{p}^p-\frac{\lambda t_\epsilon^q}{q}\|u_\epsilon\|_q^{q}\nonumber\\
=:& I_\epsilon(t_\epsilon)+B_\epsilon-C_\epsilon,
\end{align*}
where $ I_\epsilon(t):=\frac{a}{p}{A_\epsilon}t^{p}  + \frac{b}{{\theta p}}A_\epsilon^\theta t^{\theta p}  -  \frac{D_\epsilon}{p_s^*}t^{p_s^*}$. Therefore,
\begin{equation*}
c_\epsilon\leq\mathop {\sup }\limits_{t\geq 0} {I_\epsilon }(t)+B_\epsilon-C_\epsilon.
\end{equation*}
Gathering the estimates in Lemma \ref{b}, we get
\begin{eqnarray}\label{ec}
c_\varepsilon\leq \mathop {\sup }\limits_{t\geq 0} {I_\epsilon }(t)+
\left\{ \arraycolsep=1.5pt
   \begin{array}{ll}
O(\epsilon^{\frac{N-sp}{p-1}})-\frac{\lambda t_1^q}{q}\|u_\epsilon\|_q^q\ \ & \text{if}\ {sp^2>N}\\
O(\epsilon^{sp}(1+|\log\epsilon|))-\frac{\lambda t_1^q}{q}\|u_\epsilon\|_q^q\ \ & \text{if}\ {N=sp^2}\\
O(\epsilon^{sp})-\frac{\lambda t_1^q}{q}\|u_\epsilon\|_q^q\ \ & \text{if}\ {sp^2<N}.
\end{array} \right.
\end{eqnarray}
Since ${\frac{N-sp}{p-1}}>0$ and $sp>0$, we obtain
\begin{equation*}
\mathop {\sup }\limits_{t\geq 0} {I_\epsilon }(t)\geq\frac{c_\varepsilon}{2}\ \ \rm{uniformly}\ \ \rm{for}\ \ \epsilon>0\ \ \rm{small}.
\end{equation*}
Arguing as above, there exists $t_3, t_4>0$ (independent of $\epsilon>0$) such that
\begin{equation*}
\mathop {\sup }\limits_{t\geq 0} {I_\epsilon }(t)=\mathop {\sup }\limits_{t\in[t_3,t_4]} {I_\epsilon }(t).
\end{equation*}
By Lemmas \ref{a} and \ref{kmax}, we deduce
\begin{eqnarray*}
c_\varepsilon\leq&  \sup\limits_{t\geq 0} H(S^{-\frac{sp-N}{sp^2}}t)+
\left\{ \arraycolsep=1.5pt
   \begin{array}{ll}
O(\epsilon^{\frac{N-sp}{p-1}})-\frac{\lambda t_1^q}{q}\|u_\epsilon\|_q^q,\ \ & \text{if}\ {sp^2>N},\\
O(\epsilon^{sp}(1+|\log\epsilon|))-\frac{\lambda t_1^q}{q}\|u_\epsilon\|_q^q,\ \ & \text{if}\ {N=sp^2},\\
O(\epsilon^{sp})-\frac{\lambda t_1^q}{q}\|u_\epsilon\|_q^q,\ \ & \text{if}\ {sp^2<N},
\end{array} \right.\\
\leq & H(\hat T)+\left\{ \arraycolsep=1.5pt
   \begin{array}{ll}
O(\epsilon^{\frac{N-sp}{p-1}})-\frac{\lambda t_1^q}{q}\|u_\epsilon\|_q^q,\ \ & \text{if}\ {sp^2>N},\\
O(\epsilon^{sp}(1+|\log\epsilon|))-\frac{\lambda t_1^q}{q}\|u_\epsilon\|_q^q,\ \ & \text{if}\ {N=sp^2},\\
O(\epsilon^{sp})-\frac{\lambda t_1^q}{q}\|u_\epsilon\|_q^q,\ \ & \text{if}\ {sp^2<N}.
\end{array} \right.
\end{eqnarray*}

{\bf{Case 1}.} $N>sp^2$. Then $q>\theta p>p>\frac{N(p-1)}{N-sp}$ and using Lemma \ref{b}, we have
\begin{align*}
c_\varepsilon\leq H(\hat T)+O(\epsilon^{sp})-O(\lambda\epsilon^{N-\frac{N-sp}{p}q}).
\end{align*}
Using the fact $q> p$, we have $N-\frac{N-sp}{p}q<sp$. Thus, for $\epsilon>0$ small enough, we can infer that
\begin{align*}
c_\varepsilon<H(\hat T).
\end{align*}

{\bf{Case 2}.} $N=sp^2$. Thus $q>\theta p>p=\frac{N(p-1)}{N-sp}$, we have
\begin{align*}
c_\varepsilon\leq H(\hat T)+O(\epsilon^{sp}(1+|\log\epsilon|))-O(\lambda\epsilon^{sp^2-s(p-1)q}|\log\epsilon|).
\end{align*}
Observing that $q>p$ yields
\begin{align*}
\mathop {\lim }\limits_{\epsilon\to 0}\frac{\epsilon^{sp^2-s(p-1)q}|\log\epsilon|}{\epsilon^{sp}(1+|\log\epsilon|)}=\infty,
\end{align*}
we get the conclusion for $\epsilon$ small enough.

{\bf{Case 3}.} $N<sp^2$. Then we can deduce that $p_s^*-\frac{p}{p-1}>\frac{N(p-1)}{N-sp}$. Suppose that $p_s^*>q>p_s^*-\frac{p}{p-1}$ and using Lemma \ref{b}, we have
\begin{align*}
c_\varepsilon\leq H(\hat T)+O(\epsilon^{\frac{N-sp}{p-1}})-O(\lambda\epsilon^{N-\frac{N-sp}{p}q}).
\end{align*}
Using the fact $q> p_s^*-\frac{p}{p-1}$, we have $N-\frac{N-sp}{p}q<{\frac{N-sp}{p-1}}$. Thus, for $\epsilon>0$ small enough, one has
\begin{align*}
c_\varepsilon<H(\hat T).
\end{align*}
In view of $q>\theta p$, we have to compare $\theta p$ and $p_s^*-\frac{p}{p-1}$. Thus we define
\begin{equation*}
h(\theta)=\theta p-p_s^*+\frac{p}{p-1},\ \  \forall 1<\theta<\frac{N}{N-sp}.
\end{equation*}
We can easily conclude that $h(\theta)\geq0$ if $\theta\in\left[\theta_0,\frac{N}{N-sp}\right)$ and $h(\theta)< 0$ if $\theta\in (1,\theta_0)$, where $\theta_0:=\frac{N}{N-sp}-\frac{1}{p-1}$.
Consequently, let $\theta\geq\theta_0$. Then $h(\theta)\geq 0$, which implies $q>\theta p\geq p_s^*-\frac{p}{p-1}$. Arguing as above, we have the conclusion for $\epsilon>0$ small enough.

Next, we consider the case $\theta<\theta_0$. Then $h(\theta)<0$ implies $\theta p<p_s^*-\frac{p}{p-1}$. Now, we deal with the case $\theta p<q\leq p_s^*-\frac{p}{p-1}$. For this purpose, we firstly compare $\theta p$ and $\frac{N(p-1)}{N-sp}$. Then we define
\begin{equation*}
m(\theta)=\theta p-\frac{N(p-1)}{N-sp},\ \  \forall 1<\theta<\theta_0.
\end{equation*}
Since $N<sp^2$, we have $\theta_1<\theta_0$, where $\theta_1:=\frac{N(p-1)}{p(N-sp)}$.
So we can easily deduce that $m(\theta)< 0$ if $\theta\in\left(1,\theta_1\right)$ and $m(\theta)\geq0$ if $\theta\in [\theta_1,\theta_0)$.

$(i)$  If $\theta\in [\theta_1,\theta_0)$, we have $m(\theta)\geq 0$, which implies $p_s^*-\frac{p}{p-1}\geq q>\theta p\geq\frac{N(p-1)}{N-sp}$. We get
\begin{align*}
c_\varepsilon\leq H(\hat T)+O(\epsilon^{\frac{N-sp}{p-1}})-O(\lambda\epsilon^{N-\frac{N-sp}{p}q}),
\end{align*}
and choosing $\lambda=\epsilon^{-\mu}$, with $\mu>N-\frac{N-sp}{p}q-\frac{N-sp}{p-1}$, we have the conclusion.

$(ii)$ If $\theta\in\left(1,\theta_1\right)$, we have $m(\theta)< 0$, which implies $\theta p<\frac{N(p-1)}{N-sp}$. We distinguish the following case
\begin{align*}
\theta p<q<\frac{N(p-1)}{N-sp},\ \ q=\frac{N(p-1)}{N-sp}\ \ \mbox{and}\ \ \frac{N(p-1)}{N-sp}<q\leq p_s^*-\frac{p}{p-1}.
\end{align*}
When $\theta p<q<\frac{N(p-1)}{N-sp}$, then, for $\epsilon>0$ small, it holds
\begin{align*}
c_\varepsilon\leq H(\hat T)+O(\epsilon^{\frac{N-sp}{p-1}})-O(\lambda\epsilon^{\frac{N-sp}{p(p-1)}q}),
\end{align*}
and noting that
\begin{align*}
{\frac{N-sp}{p-1}}<{\frac{N-sp}{p(p-1)}q},
\end{align*}
we can take $\lambda=\epsilon^{-\mu}$, with $\mu>{\frac{N-sp}{p(p-1)}q}-{\frac{N-sp}{p-1}}$, to obtain the thesis.\\
When $q=\frac{N(p-1)}{N-sp}$, then
\begin{align*}
c_\varepsilon\leq H(\hat T)+O(\epsilon^{\frac{N-sp}{p-1}})-O(\lambda\epsilon^{\frac{N}{p}}|\log\epsilon|),
\end{align*}
and taking $\lambda=\epsilon^{-\mu}$, with $\mu>\frac{N}{p}-{\frac{N-sp}{p-1}}$, we can deduce the assertion.\\
Finally, when $\frac{N(p-1)}{N-sp}<q\leq p_s^*-\frac{p}{p-1}$, arguing as $(i)$, we get the conclusion.
\end{proof}

\begin{lemma}\label{bdd}
Every $(PS)_c$ sequence of $J_\varepsilon$ is bounded in $\mathcal{W}_\varepsilon$.
\end{lemma}
\begin{proof}
Let $\{u_n\}$ be a $(PS)$ sequence at level $c$, that is
\begin{align*}
J_\varepsilon(u_n)\rightarrow c\ \ \text{and}\ \ {{J'}_\varepsilon }({u_n})\rightarrow 0,
\end{align*}
as $n\to\infty$.
Using $(G_3)$, we can deduce that
\begin{align*}
c+o_n(1)=&J_\varepsilon(u_n)-\frac{1}{\mu}\left\langle {{{J'}_\varepsilon }({u_n}),{u_n}} \right\rangle\\
=&\left({\frac{{\mu - p}}{{p\mu }}}\right)\|u_n\|_\varepsilon^p+\frac{1}{\mu}\int_{\mathbb{R}^N\setminus{\Lambda_\varepsilon}} {[g(\varepsilon x,{u_n}){u_n} - \mu G} (\varepsilon x,{u_n})]dx\\
&+\frac{1}{\mu}\int_{{\Lambda_\varepsilon}} {[g(\varepsilon x,{u_n}){u_n} - \mu G} (\varepsilon x,{u_n})]dx+b(\frac{\mu-\theta p}{\theta p\mu})[{u_n}]_{s,p}^{2p}\\
\geq& \left({\frac{{\mu  - p}}{{p\mu }}}\right)\|u_n\|_\varepsilon^p+\frac{1}{\mu}\int_{\mathbb{R}^N\setminus{\Lambda_\varepsilon}} {[g(\varepsilon x,{u_n}){u_n} - \mu G} (\varepsilon x,{u_n})]dx\\
\geq& \left({\frac{{\mu  - p}}{{p\mu }}}\right)\|u_n\|_\varepsilon^p-\frac{\mu-p}{\mu}\int_{\mathbb{R}^N\setminus{\Lambda_\varepsilon}} { G} (\varepsilon x,{u_n})dx\\
\geq&\left({\frac{{\mu  - p}}{{p\mu }}}\right)\|u_n\|_\varepsilon^p-\left(\frac{\mu-p}{p\mu}\right)\int_{\mathbb{R}^N\setminus{\Lambda_\varepsilon}} \frac{V(\varepsilon x)}{K}|u_n|^pdx\\
\geq& \left({\frac{{\mu - p}}{{p\mu }}}\right)\left(1-\frac{1}{K}\right)\|u_n\|_\varepsilon^p.
\end{align*}
Since $\mu>\theta p$ and $K>2$, we can conclude that $\{u_n\}_{n\in\mathbb{N}}$ is bounded in $\mathcal{W}_\varepsilon$.
\end{proof}

\begin{lemma}\label{nonv}
There exists a sequence $\{z_n\}_{n\in\mathbb{N}}\subset\mathbb{R}^N$ and $R, \beta>0$ such that
\begin{equation*}
\int_{B_R(z_n)} |u_n|^p\ dx\geq\beta>0.
\end{equation*}
Moreover, $\{z_n\}_{n\in\mathbb{N}}$ is bounded in $\mathbb{R}^N$.
\end{lemma}
\begin{proof}
By Lemma \ref{bdd}, we know that $\{u_n\}$ is bounded in $\mathcal{W}_\varepsilon$. Then we may assume that $u_n\rightharpoonup u$ in $\mathcal{W}_\varepsilon$. Below we first prove that $\{u_n\}$ is non-vanishing. Otherwise, assume that $\{u_n\}$ is vanishing. Then Lemma \ref{Lion} implies that $u_n\to 0$ in $L^q(\mathbb{R}^N)$ for any $q\in(p,p_s^*)$. From $(F_1)$ and $(F_2)$, it follows that
\begin{equation*}
\int_{\mathbb{R}^N} {F({u_n})} dx=\int_{\mathbb{R}^N} {f({u_n}){u_n}} dx=o_n(1)\ \ \text{as}\ \ n\to\infty.
\end{equation*}
This implies that
\begin{equation}\label{G_un}
\int_{\mathbb{R}^N} G(\varepsilon x,u_n) dx=\frac{1}{{p_s^*}}\int_{{\Lambda _\varepsilon } \cup \{ {u_n} \le a\} } {{{(u_n^ + )}^{p_s^*}}} dx + \frac{{{V_0}}}{{pK}}\int_{(\mathbb{R}^N\setminus\Lambda_\varepsilon)\cap \{ {u_n} > a\}} {|{u_n}{|^p}} dx + {o_n}(1)
\end{equation}
and
\begin{equation}\label{g_un}
\int_{\mathbb{R}^N} g(\varepsilon x,u_n)u_n dx=\int_{{\Lambda _\varepsilon } \cup \{ {u_n} \le a\} } {{{(u_n^ + )}^{p_s^*}}} dx + \frac{{{V_0}}}{{K}}\int_{(\mathbb{R}^N\setminus\Lambda_\varepsilon)\cap \{ {u_n}> a\}} {|{u_n}{|^p}} dx + {o_n}(1).
\end{equation}
Using $\left\langle {{J'_\varepsilon }({u_n}),{u_n}} \right\rangle=o_n(1)$ and (\ref{g_un}) we have
\begin{equation}\label{dJ}
\|u_n\|_\varepsilon^p+b[u_n]_{s,p}^{\theta p}-\frac{{{V_0}}}{{K}}\int_{(\mathbb{R}^N\setminus\Lambda_\varepsilon)\cap \{ {u_n} > a\}} {|{u_n}{|^p}} dx =\int_{{\Lambda _\varepsilon } \cup \{ {u_n} \le a\} } {{{(u_n^ + )}^{p_s^*}}} dx + {o_n}(1).
\end{equation}
Suppose that
\begin{equation*}
 \int_{{\Lambda _\varepsilon } \cup \{ {u_n} \le a\} } {{{(u_n^ + )}^{p_s^*}}} dx\rightarrow l^{p_s^*}.
\end{equation*}
Note that $l>0$, otherwise (\ref{dJ}) yields $\|u_n\|_\varepsilon\rightarrow 0$ as $n\to\infty$ which implies that $J_\varepsilon(u_n)\rightarrow 0$, and this is a contradiction because $c_\varepsilon>0$. Then, by (\ref{dJ}) and Sobolev inequality we obtain
\begin{equation}\label{Sob}
a{S} \left( \int_{{\Lambda _\varepsilon } \cup \{ {u_n} \le a\} } {{{(u_n^ + )}^{p_s^*}}} dx\right)^{\frac{p}{p_s^*}}+ bS^\theta \left( \int_{{\Lambda _\varepsilon } \cup \{ {u_n} \le a\} } {{{(u_n^ + )}^{p_s^*}}} dx\right)^{\frac{\theta p}{p_s^*}}\leq \int_{{\Lambda _\varepsilon } \cup \{ {u_n} \le a\} } {{{(u_n^ + )}^{p_s^*}}} dx+o_n(1).
\end{equation}
Since $l>0$, by (\ref{Sob}), we have
\begin{equation*}
a{S} l^{p}+ bS^\theta l^{\theta p}-l^{p_s^*} \leq 0,
\end{equation*}
which implies $H'(l)\leq 0$ defined in (\ref{dK}), so we can deduce that $l\geq {\hat T}$, where ${\hat T}$ is the unique maximum of $H$ defined in Lemma \ref{kmax}. And we can know that
\begin{equation*}
\hat H({\hat T}) = a{S} + bS^\theta{{\hat T}^{(\theta-1)p}} - {{\hat T}^{p_s^* - p}}=0,
\end{equation*}
which implies
\begin{equation}\label{T1}
 a{S}{\hat T}^p + bS^\theta {{\hat T}^{\theta p}} - {{\hat T}^{p_s^* }}=0.
\end{equation}
Therefore, from (\ref{G_un}), (\ref{g_un}), (\ref{Sob}), $l\geq{\hat T}$ and Sobolev inequality we can infer
\begin{align*}
c+o_n(1)=&J_\varepsilon(u_n)- \frac{{1}}{{\theta p}}\left\langle {{{J'}_\varepsilon }({u_n}),{u_n}} \right\rangle\\
=& {\frac{(\theta-1)a}{\theta p}}[u_n]_{s,p}^p+ \frac{{(\theta-1)}}{{\theta p}}\int_{\mathbb{R}^N} {V(\varepsilon x)|{u_n}|^p} dx+ \frac{{1}}{{\theta p}}\int_{\mathbb{R}^N} g(\varepsilon x,u_n)u_n dx-\int_{\mathbb{R}^N} G(\varepsilon x,u_n)dx\\
\geq& {\frac{(\theta-1)a}{\theta p}}[u_n]_{s,p}^p+ \frac{{(\theta-1)}}{{ \theta p}}\int_{\mathbb{R}^N} {V(\varepsilon x)|{u_n}|^p} dx-\frac{{(\theta-1)V_0}}{{\theta pK}}\int_{\mathbb{R}^N} |{u_n}|^p dx\\
&+\left(\frac{{1}}{{\theta p}}-\frac{1}{{p_s^*}}\right)\int_{{\Lambda _\varepsilon } \cup \{ {u_n} \le a\} } {{{(u_n^ + )}^{p_s^*}}} dx \\
\geq&\frac{(\theta-1)a}{\theta p}{S} \left( \int_{{\Lambda _\varepsilon } \cup \{ {u_n} \le a\} } {{{(u_n^ + )}^{p_s^*}}} dx\right)^{\frac{p}{p_s^*}}\\
&+\left(\frac{{1}}{{\theta p}}-\frac{1}{{p_s^*}}\right)\left(a{S} \left( \int_{{\Lambda _\varepsilon } \cup \{ {u_n} \le a\} } {{{(u_n^ + )}^{p_s^*}}} dx\right)^{\frac{p}{p_s^*}}+ bS^\theta\left( \int_{{\Lambda _\varepsilon } \cup \{ {u_n} \le a\} } {{{(u_n^ + )}^{p_s^*}}} dx\right)^{\frac{\theta p}{p_s^*}}\right)\\
=& \frac{(\theta-1)a}{\theta p}{S} l^{p}+\left(\frac{{1}}{{\theta p}}-\frac{1}{{p_s^*}}\right)\left(a{S} l^{p}+ bS^\theta l^{\theta p}\right)\\
\geq&\frac{(\theta-1)a}{\theta p}{S} {\hat T}^{p}+\left(\frac{{1}}{{\theta p}}-\frac{1}{{p_s^*}}\right)\left(a{S} {\hat T}^{p}+ bS^\theta{\hat T}^{\theta p}\right)\\
=& \frac{a}{p}{S} {\hat T}^{p}+ \frac{b}{\theta p}S^\theta{\hat T}^{\theta p}-\frac{1}{p_s^*} {\hat T}^{p_s^*},
\end{align*}
which contradicts with Lemma \ref{compare}.\

Now, we show that $\{z_n\}$ is bounded in $\mathbb{R}^N$. Let us consider the function $\eta_R\in \mathcal{C}^\infty(\mathbb{R}^N)$ defined as $\eta_R(x)=0$ if $x\in B_R$, $\eta_R(x)=1$ if $x \notin  B_{2R}$, with $0\leq\eta_R\leq1$ and $|\nabla\eta_R|\leq \frac{C}{R}$, where $C$ is a constant independent of $R$. Since $\{\eta_Ru_n\}$ is bounded in $\mathcal{W}_\varepsilon$, $\left\langle {{J'_\varepsilon }({u_n}),{u_n}\eta_R} \right\rangle=o_n(1)$, we get
\begin{align*}
&(a + b[{u_n}]_{s,p}^{(\theta-1)p})\iint_{\mathbb{R}^{2N}}{\frac{{|{u_n}(x) - {u_n}(y){|^p}}}{{|x - y{|^{N + sp}}}}\eta_R(x)}dxdy+\int_{\mathbb{R}^N} {V(\varepsilon x)|{u_n}|^p} \eta_R(x)dx\\
=&o_n(1)+\int_{\mathbb{R}^N} g(\varepsilon x,u_n)u_n\eta_R(x) dx\\
&-(a + b[{u_n}]_{s,p}^{(\theta-1)p})\iint_{\mathbb{R}^{2N}} \frac{{|{u_n}(x) - {u_n}(y){|^{p - 2}}({u_n}(x) - {u_n}(y))({\eta _R}(x) - {\eta _R}(y))}}{{|x - y{|^{N + sp}}}}u_n(y)dxdy.
\end{align*}
Take $R>0$ such that $\Lambda_\varepsilon\subset B_R$. Then, using $(G_3)$-$(ii)$, we get
\begin{align}\label{*1}
&\iint_{\mathbb{R}^{2N}}{a\frac{{|{u_n}(x) - {u_n}(y){|^p}}}{{|x - y{|^{N+ sp}}}}\eta_R(x)}dxdy+\int_{\mathbb{R}^N} {V(\varepsilon x)|{u_n}|^p} \eta_R(x)dx\nonumber\\
\leq &o_n(1)+\int_{\mathbb{R}^N} {\frac{V(\varepsilon x)}{K}|{u_n}|^p \eta_R(x)} dx\nonumber\\
&-(a + b[{u_n}]_{s,p}^{(\theta-1)p})\iint_{\mathbb{R}^{2N}} \frac{{|{u_n}(x) - {u_n}(y){|^{p - 2}}({u_n}(x)- {u_n}(y))({\eta _R}(x) - {\eta _R}(y))}}{{|x - y{|^{N + sp}}}}u_n(y)dxdy,
\end{align}
which implies that
\begin{align*}
&(1-\frac{1}{K})V_0\int_{\mathbb{R}^N} |{{u_n}|^p} \eta_R(x)dx\\
\leq &o_n(1)-(a + b[{u_n}]_{s,p}^{(\theta-1)p})\iint_{\mathbb{R}^{2N}} \frac{{|{u_n}(x) - {u_n}(y){|^{p - 2}}({u_n}(x) - {u_n}(y))({\eta _R}(x) - {\eta _R}(y))}}{{|x - y{|^{N + sp}}}}u_n(y)dxdy.
\end{align*}
Exploiting the H\"{o}lder inequality and the boundedness on $\{u_n\}$ in $\mathcal{W}_\varepsilon$ we have that
\begin{align}\label{*2}
&\left|\iint_{\mathbb{R}^{2N}} \frac{{|{u_n}(x) - {u_n}(y){|^{p - 2}}({u_n}(x) - {u_n}(y))({\eta _R}(x) - {\eta _R}(y))}}{{|x - y{|^{N + sp}}}}u_n(y)dxdy\right|\nonumber\\
\leq &C\left(\iint_{\mathbb{R}^{2N}} \frac{{|{\eta _R}(x) - {\eta _R}(y){|^{p }}}}{{|x - y{|^{N + sp}}}}|u_n(x)|^pdxdy\right)^{\frac{1}{p}}.
\end{align}
Now, we show that
\begin{equation*}
\mathop {\lim }\limits_{R \to \infty } \mathop {\lim \sup }\limits_{n \to \infty } \iint_{\mathbb{R}^{2N}} {\frac{{|{\eta _R}(x) - {\eta _R}(y){|^p}}}{{|x - y{|^{N + sp}}}}|{u_n}(x){|^p}dxdy}=0.
\end{equation*}
In fact, recalling that $0\leq\eta_R\leq1$ and $|\nabla\eta_R|\leq \frac{C}{R}$, and using polar coordinates, we obtain
\begin{align}\label{*3}
&\int_{\mathbb{R}^N} {\frac{{|{\eta _R}(x) - {\eta _R}(y){|^p}}}{{|x - y{|^{N + sp}}}}|{u_n}(x){|^p}dxdy}\nonumber\\
= &\int_{\mathbb{R}^N} {\int_{|y - x| > R} {\frac{{|{\eta _R}(x) - {\eta _R}(y){|^p}}}{{|x - y{|^{N + sp}}}}|{u_n}(x){|^p}dxdy + } } \int_{\mathbb{R}^N} {\int_{|y - x| \le R} {\frac{{|{\eta _R}(x) - {\eta _R}(y){|^p}}}{{|x - y{|^{N + sp}}}}|{u_n}(x){|^p}dxdy} }\nonumber\\
\leq&C \int_{\mathbb{R}^N} {|{u_n}(x){|^p}\left(\int_{|y - x| > R} \frac{{dy}}{{|x - y{|^{N + sp}}}}\right) } dx+ \frac{C}{R^p}\int_{\mathbb{R}^N} |{u_n}(x){|^p}\left(\int_{|y - x| \leq R} \frac{{dy}}{{|x - y{|^{N + sp-p}}}}\right) dx\nonumber\\
\leq&C \int_{\mathbb{R}^N} {|{u_n}(x){|^p}\left(\int_{|z| > R} \frac{{dy}}{{|z{|^{N + sp}}}}\right)}  dx+ \frac{C}{R^p}\int_{\mathbb{R}^N} {|{u_n}(x){|^p}\left(\int_{|z| \leq R} \frac{{dy}}{{|z{|^{N + sp-p}}}}\right)} dx\nonumber\\
\leq&C \int_{\mathbb{R}^N} {|{u_n}(x){|^p}\left(\int_R^\infty  {\frac{{{\rho ^{N-1}}}}{{{\rho ^{N + sp}}}}} d\rho \right)}  dx+ \frac{C}{R^p}\int_{\mathbb{R}^N} {|{u_n}(x){|^p}\left(\int_0^R  {\frac{{{\rho ^{N-1}}}}{{{\rho ^{N + sp-p}}}}} d\rho \right)}  dx\nonumber\\
\leq&\frac{C}{{{R^{sp}}}}\int_{\mathbb{R}^N} |{u_n}(x)|^pdx+\frac{C}{{{R^{p}}}}R^{p-sp}\int_{\mathbb{R}^N} |{u_n}(x)|^pdx\nonumber\\
\leq&\frac{C}{{{R^{sp}}}}\int_{\mathbb{R}^N} |{u_n}(x){|^p}dx\leq\frac{C}{{{R^{sp}}}},
\end{align}
where in the last passage we used the boundedness of $\{u_n\}$ in $\mathcal{W}_\varepsilon$. From the above estimate we can deduce that
\begin{align*}
(1-\frac{1}{K})V_0\int_{\mathbb{R}^N} {|{u_n}|^p} \eta_R(x)dx\leq \frac{C}{{{R^{s}}}}+o_n(1),
\end{align*}
which implies that $\{z_n\}$ is bounded in $\mathbb{R}^3$.
\end{proof}

\begin{lemma}\label{conv}
Let $\{u_n\}$ be a $(PS)_c$ sequence for ${J_\varepsilon }$, then for each $\zeta>0$, there is a number $R_1=R(\zeta)>0$ such that
\begin{equation*}
\mathop {\lim \sup }\limits_{n \to \infty } \int_{\mathbb{R}^N\setminus{B_{{R_1}}}(0)} \left(a\int_{\mathbb{R}^N} {\frac{{|{u_n}(x) - {u_n}(y){|^p}}}{{|x - y{|^{N + sp}}}}}  dy + V(\varepsilon x)u_n^p\right)dx <\zeta.
\end{equation*}
\end{lemma}
\begin{proof}
Let $\eta_R\in\mathcal{C}^\infty(\mathbb{R}^N)$ be a cut-off function such that $\eta_R(x)=0$ in $B_R$ and $\eta_R(x)=1$ in $\mathbb{R}^N\setminus B_{2R}$, with $0\leq\eta_R(x)\leq1$, $|\nabla\eta_R|\leq \frac{C}{R}$ and $C$ is a constant independent on $R$. Take $R>0$ such that $\Lambda_\varepsilon\subset B_R(0)$. In view of Lemma \ref{bdd}, $\{u_n\}$ is bounded. From $\left\langle {{J'_\varepsilon }({u_n}),{u_n}\eta_R} \right\rangle=o_n(1)$ and (\ref{*1}), we have
\begin{align}\label{*4}
&\iint_{\mathbb{R}^{2N}}{a\frac{{|{u_n}(x) - {u_n}(y){|^p}}}{{|x - y{|^{N + sp}}}}\eta_R(x)}dxdy+\left(1-\frac{1}{K}\right)\int_{\mathbb{R}^N} {V(\varepsilon x)|{u_n}|^p} \eta_R(x)dx\nonumber\\
\leq &o_n(1)-(a + b[{u_n}]_{s,p}^{(\theta-1)p})\iint_{\mathbb{R}^{2N}} \frac{{|{u_n}(x) - {u_n}(y){|^{p - 2}}({u_n}(x)- {u_n}(y))({\eta _R}(x) - {\eta _R}(y))}}{{|x - y{|^{N + sp}}}}u_n(y)dxdy.
\end{align}
Gathering (\ref{*2}), (\ref{*3}) and (\ref{*4}) we infer that
\begin{equation*}
\mathop {\lim \sup }\limits_{n \to \infty } \int_{\mathbb{R}^{N}\setminus{B_{{R_1}}}(0)} \left(a\int_{\mathbb{R}^N} {\frac{{|{u_n}(x) - {u_n}(y){|^p}}}{{|x - y{|^{N + sp}}}}}  dy + V(\varepsilon x)u_n^p\right)dx <\zeta.
\end{equation*}
\end{proof}

\begin{lemma}\label{conv1}
Let $\{u_n\}$ be a $(PS)_c$ sequence for ${J_\varepsilon }$ such that $u_n\rightharpoonup u$ in ${\mathcal{W}}_\varepsilon$, then
\begin{align*}
&\mathop {\lim }\limits_{n \to \infty } \int_{{B_{{R}}}(0)} \left(a\int_{\mathbb{R}^N} {\frac{{|{u_n}(x) - {u_n}(y){|^p}}}{{|x - y{|^{N + sp}}}}}  dy + V(\varepsilon x)u_n^p\right)dx \\
=&\int_{{B_{{R}}}(0)} \left(a\int_{\mathbb{R}^N} {\frac{{|{u}(x) - {u}(y){|^p}}}{{|x - y{|^{N + sp}}}}}  dy + V(\varepsilon x)u^p\right)dx.
\end{align*}
\end{lemma}
\begin{proof}
We define the cut-off function $\eta_\rho\in \mathcal{C}^\infty(\mathbb{R}^N)$ such that $\eta_\rho(x)=0$ in $B_\rho$ and $\eta_\rho(x)=1$ in $\mathbb{R}^N\setminus B_{2\rho}$, with $0\leq\eta_\rho(x)\leq1$. Set
\begin{align*}
{P_n}(x): &= {M_n}\int_{\mathbb{R}^N} {\frac{\left[|u_n(x)-{u_n}(y)|^{p-2}(u_n(x)-{u_n}(y))-|u(x)-u(y)|^{p-2}(u(x)-u(y))\right]}{{|x - y{|^{N + sp}}}}}\\
 &\times{[({u_n}(x) - u(x)) - ({u_n}(y) - u(y))]}dy + V(\varepsilon x)\left(|u_n|^{p-2}u_n-|u|^{p-2}u\right)(u_n-u),
\end{align*}
where $M_n:=a+b[u_n]_{s,p}^{(\theta-1)p}$.\

Fix $R>0$ and choose $R<\rho$. We have
\begin{align*}
&{M_n}\iint_{\mathbb{R}^{2N}} {\frac{\left[|u_n(x)-{u_n}(y)|^{p-2}(u_n(x)-{u_n}(y))-|u(x)-u(y)|^{p-2}(u(x)-u(y))\right]}{{|x - y{|^{N + sp}}}}}\\
 &\times{[({u_n}(x) - u(x)) - ({u_n}(y) - u(y))]}\eta_\rho dxdy + \int_{\mathbb{R}^N} V(\varepsilon x)\left(|u_n|^{p-2}u_n-|u|^{p-2}u\right)(u_n-u)\eta_\rho dx\\
=&{M_n}\iint_{\mathbb{R}^{2N}} {\frac{{|{u_n}(x) -{u_n}(y){|^p}}}{{|x - y{|^{N + sp}}}}}\eta_\rho(x) dxdy+ \int_{\mathbb{R}^N}V(\varepsilon x)|{u_n}{|^p}\eta_\rho dx\\
&+{M_n}\iint_{\mathbb{R}^{2N}} {\frac{{|{u}(x) -{u}(y){|^p}}}{{|x - y{|^{N + sp}}}}}\eta_\rho(x) dxdy+ \int_{\mathbb{R}^N}V(\varepsilon x)|{u}{|^p}\eta_\rho dx\\
&-\left[\left({M_n}\iint_{\mathbb{R}^{2N}} {\frac{{|{u_n}(x) -{u_n}(y){|^{p-2}}}(u_n(x)-u_n(y))(u(x)-u(y))}{{|x - y{|^{N + sp}}}}}\eta_\rho(x) dxdy+ \int_{\mathbb{R}^N}V(\varepsilon x)|{u_n}{|^{p-2}}{u_n}u\eta_\rho dx\right)\right.\\
&\left.-\left({M_n}\iint_{\mathbb{R}^{2N}} {\frac{{|{u}(x) -{u}(y){|^{p-2}}}(u(x)-u(y))(u_n(x)-u_n(y))}{{|x - y{|^{N + sp}}}}} \eta_\rho(x) dxdy +\int_{\mathbb{R}^N}V(\varepsilon x)|{u}{|^{p-2}}u u_n\eta_\rho dx\right)\right].\\
\end{align*}
Setting
\begin{align*}
&I_{n,\rho}^1:={M_n}\iint_{\mathbb{R}^{2N}} {\frac{{|{u_n}(x) -{u_n}(y){|^p}}}{{|x - y{|^{N + sp}}}}}\eta_\rho(x) dxdy+ \int_{\mathbb{R}^N}V(\varepsilon x)|{u_n}{|^p}\eta_\rho dx-\int_{\mathbb{R}^N}g(\varepsilon x,u_n)u_n\eta_\rho dx.\\
&I_{n,\rho}^2:={M_n}\iint_{\mathbb{R}^{2N}} {\frac{{|{u_n}(x) -{u_n}(y){|^{p-2}}}(u_n(x)-u_n(y))(u(x)-u(y))}{{|x - y{|^{N + sp}}}}}\eta_\rho(x) dxdy\\
&\qquad\qquad\qquad+\int_{\mathbb{R}^N}V(\varepsilon x)|{u_n}{|^{p-2}}{u_n}u\eta_\rho dx-\int_{\mathbb{R}^N}g(\varepsilon x,u_n)u\eta_\rho dx.\\
&I_{n,\rho}^3:=\left({M_n}\iint_{\mathbb{R}^{2N}} {\frac{|u(x)-u(y)|^p}{{|x - y{|^{N + sp}}}}}\eta_\rho(x) dxdy+ \int_{\mathbb{R}^N}V(\varepsilon x)|u|^p\eta_\rho dx\right)\\
&-\left({M_n}\iint_{\mathbb{R}^{2N}} {\frac{{|{u}(x) -{u}(y){|^{p-2}}}(u(x)-u(y))(u_n(x)-u_n(y))}{{|x - y{|^{N + sp}}}}} \eta_\rho(x) dxdy +\int_{\mathbb{R}^N}V(\varepsilon x)|{u}{|^{p-2}}u u_n\eta_\rho dx\right)\\
&I_{n,\rho}^4:=\int_{\mathbb{R}^N}g(\varepsilon x,u_n)(u_n-u)\eta_\rho dx,
\end{align*}
obviously, we have
\begin{align*}
0\leq\int_{B_R(0)}{P_n}(x)dx\leq |I_{n,\rho}^1|+|I_{n,\rho}^2|+|I_{n,\rho}^3|+|I_{n,\rho}^4|.
\end{align*}
Next, we have to prove that
\begin{align*}
\mathop {\lim }\limits_{\rho  \to \infty } [\mathop {\lim \sup }\limits_{n \to \infty } |I_{n,\rho}^1|]=0,\qquad\qquad\mathop {\lim }\limits_{\rho  \to \infty } [\mathop {\lim \sup }\limits_{n \to \infty } |I_{n,\rho}^2|]=0
\end{align*}
and
\begin{align*}
\mathop {\lim }\limits_{\rho  \to \infty } [\mathop {\lim \sup }\limits_{n \to \infty } |I_{n,\rho}^3|]=0,\qquad\qquad\mathop {\lim }\limits_{\rho  \to \infty } [\mathop {\lim \sup }\limits_{n \to \infty } |I_{n,\rho}^4|]=0.
\end{align*}
Observe that
\begin{align*}
I_{n,\rho}^1=J'_\varepsilon(u_n)(u_n\eta_\rho)-(a+b[u_n]_{s,p}^{(\theta-1)p})\iint_{\mathbb{R}^{2N}} {\frac{{|{u_n}(x) -{u_n}(y){|^{p-2}}}(u_n(x)-u_n(y))(\eta_\rho(x)-\eta_\rho(y))}{{|x - y{|^{N + sp}}}}}u_n(x) dxdy
\end{align*}
Since $\{u_n\eta_\rho\}$ bounded in $W^{s,p}(\mathbb{R}^N)$, we have $\left\langle {{J'_\varepsilon }({u_n}),{u_n}\eta_R} \right\rangle=o_n(1)$. Moreover, by H\"{o}lder inequality and the proof of (\ref{*3}), one has
\begin{align*}
&(a+b[u_n]_{s,p}^{(\theta-1)p})\iint_{\mathbb{R}^{2N}} {\frac{{|{u_n}(x) -{u_n}(y){|^{p-2}}}(u_n(x)-u_n(y))(\eta_\rho(x)-\eta_\rho(y))}{{|x - y{|^{N + sp}}}}}u_n(x) dxdy\\
\leq& C_1\left(\iint_{\mathbb{R}^{2N}} \frac{{|{\eta _\rho}(x) - {\eta _\rho}(y){|^{p }}}}{{|x - y{|^{N + sp}}}}|u_n(x)|^pdxdy\right)^{\frac{1}{p}}\left(\iint_{\mathbb{R}^{2N}} \frac{{|u_n(x) - u_n(y){|^{p }}}}{{|x - y{|^{N+ sp}}}}dxdy\right)^{\frac{p-1}{p}}\\
\leq&C_2\left(\iint_{\mathbb{R}^{2N}} \frac{{|(1-{\eta _\rho}(x)) - (1-{\eta _\rho}(y)){|^{p }}}}{{|x - y{|^{N + sp}}}}|u_n(x)|^pdxdy\right)^{\frac{1}{p}}\rightarrow 0
\end{align*}
as $\rho\to\infty$, $n\to\infty$. Then,
\begin{align*}
\mathop {\lim }\limits_{\rho  \to \infty } [\mathop {\lim \sup }\limits_{n \to \infty } |I_{n,\rho}^1|]=0.
\end{align*}
We also have that
\begin{align*}
I_{n,\rho}^2=J'_\varepsilon(u_n)(u\eta_\rho)-(a+b[u_n]_{s,p}^{(\theta-1)p})\iint_{\mathbb{R}^{2N}} {\frac{{|{u_n}(x) -{u_n}(y){|^{p-2}}}(u_n(x)-u_n(y))(\eta_\rho(x)-\eta_\rho(y))}{{|x - y{|^{N + sp}}}}}u(x) dxdy.
\end{align*}
Using $\left\langle {{J'_\varepsilon }({u_n}),{u}\eta_R} \right\rangle=o_n(1)$ and arguing as in the previous case, we can prove the second term of the right side of the last equality tends to zero, as $n, \rho\to\infty$. Consequently,
\begin{align*}
\mathop {\lim }\limits_{\rho  \to \infty } [\mathop {\lim \sup }\limits_{n \to \infty } |I_{n,\rho}^2|]=0.
\end{align*}
On the other hand, from the weak convergence,
\begin{align*}
\mathop {\lim }\limits_{\rho  \to \infty } |I_{n,\rho}^3|=&\mathop {\lim }\limits_{\rho  \to \infty }\left|\left({M_n}\iint_{\mathbb{R}^{2N}} {\frac{|u(x)-u(y)|^{p-2}(u(x)-u(y))[(u_n(x)-u_n(y))-(u(x)-u(y))]}{{|x - y{|^{N + sp}}}}}\eta_\rho(x) dxdy\right.\right.\\
&\left.\left.+ \int_{\mathbb{R}^N}V(\varepsilon x)|u|^{p-2}u(u_n-u)\eta_\rho dx\right)\right|=0,
\end{align*}
where we used the fact that $M_n$ is bounded in $\mathbb{R}$.
Below we shall prove that
\begin{align*}
\mathop {\lim }\limits_{\rho  \to \infty } [\mathop {\lim \sup }\limits_{n \to \infty } |I_{n,\rho}^4|]=0.
\end{align*}
We first claim that
\begin{align*}
{u_n} \to u\ \ \text{in}\ \  L^{p_s^*}(\Lambda_\varepsilon).
\end{align*}
Indeed, we may suppose that $|D^su_n|^p\rightharpoonup \mu$ and $|u_n|^{p_s^*}\rightharpoonup\nu$, where $|D^su_n|^p=\int_{\mathbb{R}^N} \frac{|u(x)-u(y)|^p}{|x-y|^{N+sp}}dy$, $\mu$ and $\nu$ are bounded nonnegative measures in ${\mathbb{R}^N}$. By a variant of the concentration-compactness lemma \cite{Lion1}, which is proved by {{\cite[Lemma 2.4]{VA}}}, we have an at most countable index set $\Gamma$ and sequence $\{x_i\}\subset{\mathbb{R}^N}$, $\{\mu_i\}$, $\{\nu_i\}\subset (0,\infty)$ such that
\begin{align}\label{ccl}
\mu  \ge |{D^s}u{|^p} + \sum\limits_{i \in \Gamma} {{\mu _i}} {\delta _{{x_i}}},\ \ \nu  = |u{|^{p_s^*}} + \sum\limits_{i \in \Gamma } {{\nu _i}} {\delta _{{x_i}}},\ \ {\mu _i} \ge {S}{\nu _i}^{\frac{p}{{p_s^*}}}
\end{align}
for all $i\in \Gamma$, where ${\delta _{{x_i}}}$ is the Dirac mass at $x_i\in{\mathbb{R}^N}$. It suffices to show that $\{x_i\}_{i\in\Gamma}\cap\Lambda_\varepsilon=\emptyset$. Suppose, by contradiction, that ${x_i}\in \Lambda_\varepsilon$  for some $i\in\Gamma$. Define $\varphi_\rho(x)=\varphi\left(\frac{x-x_i}{\rho}\right)$, where $\varphi=1$ if $x\in B_1(0)$, $\varphi=0$ if $x \notin B_2(0)$, and $|\nabla\varphi|_\infty\leq2$. We suppose that $\rho$ is chosen in such way that the support of $\varphi_\rho$ is contained in $\Lambda_\varepsilon$. Since $\{\varphi_\rho u_n\}$ is bounded , $\left\langle {{J'_\varepsilon }({u_n}),{u_n}\varphi_\rho} \right\rangle=o_n(1)$, we have
\begin{align*}
&(a+b[u_n]_{s,p}^{(\theta-1)p})\int_{\mathbb{R}^N}(-\Delta)_p^su_n(\varphi_\rho u_n)dx+\int_{\mathbb{R}^N}V(\varepsilon x)u_n^p\varphi_\rho dx\\
=&\int_{\mathbb{R}^N} g(\varepsilon x,u_n)u_n\varphi_\rho dx+o_n(1)\\
=&\int_{\mathbb{R}^N}f(u_n)u_n\varphi_\rho dx+\int_{\mathbb{R}^N} u_n^{p_s^*}\varphi_\rho  dx+o_n(1).
\end{align*}
Since $\varphi_\rho$ has compact support, letting $n\to\infty$ and $\rho\to 0$ we obtain that $a{\mu _i} + b\mu _i^\theta \le {\nu _i}$, which combining with inequality that $ {\mu _i} \ge {S}{\nu _i}^{\frac{p}{{p_s^*}}}$ we infer that ${\nu _i} \ge a{S}{\nu _i}^{\frac{p}{{p_s^*}}} + bS^\theta{\nu _i}^{\frac{{\theta p}}{{p_s^*}}}$. Setting ${\nu '_i}:=\nu _i^{\frac{1}{{p_s^*}}}$, we have
\begin{align*}
a{S}{({\nu '_i})^p} + bS^\theta{({\nu '_i})^{\theta p}} - {({\nu '_i})^{p_s^*}} \le 0,
\end{align*}
which implies $H'({\nu '_i})\leq 0$ defined in (\ref{dK}), so we can deduce that ${\nu '_i}\geq {\hat T}$, where ${\hat T}$ is the unique maximum of $H$ defined in Lemma \ref{kmax}. And we can know that
\begin{equation*}
\hat H({\hat T}) = a{S} + bS^\theta{{\hat T}^{(\theta-1)p}} - {{\hat T}^{p_s^* - p}}=0,
\end{equation*}
which implies
\begin{equation}\label{T}
 a{S}{\hat T}^p + bS^\theta{{\hat T}^{\theta p}} - {{\hat T}^{p_s^* }}=0.
\end{equation}
Therefore, we can infer
\begin{align*}
c+o_n(1)=&J_\varepsilon(u_n)- \frac{{1}}{{\theta p}}\left\langle {{{J'}_\varepsilon }({u_n}),{u_n}} \right\rangle\\
\geq& {\frac{(\theta-1)a}{\theta p}}[u_n]_{s,p}^p+\left(\frac{{1}}{{\theta p}}-\frac{1}{{p_s^*}}\right)\int_{{\Lambda _\varepsilon } } {{{(u_n )}^{p_s^*}}} dx+ \frac{{(\theta-1)}}{{\theta p}}\int_{\mathbb{R}^N} {V(\varepsilon x)|{u_n}|^p} dx\\
&+ \frac{{1}}{{\theta p}}\int_{{\mathbb{R}^N}\setminus{\Lambda _\varepsilon }}\left( g(\varepsilon x,u_n)u_n-\theta p G(\varepsilon x,u_n)\right) dx\\
\geq& {\frac{(\theta-1)a}{\theta p}}[u_n]_{s,p}^p+ \frac{{(\theta-1)}}{{\theta p}}\int_{\mathbb{R}^N} {V(\varepsilon x)|{u_n}|^p} dx-\frac{{(\theta-1)}}{{\theta}}\int_{{\mathbb{R}^N}\setminus{\Lambda _\varepsilon }} G(\varepsilon x,u_n) dx\\
&+\left(\frac{{1}}{{\theta p}}-\frac{1}{{p_s^*}}\right)\int_{{\Lambda _\varepsilon } } {{{(u_n )}^{p_s^*}}} dx \\
\geq&\frac{(\theta-1)a}{\theta p}\int_{\Lambda_\varepsilon}|D^su_n|^p\varphi_\rho dx+\left(\frac{{1}}{{\theta p}}-\frac{1}{{p_s^*}}\right)\int_{{\Lambda _\varepsilon } } {{{(u_n )}^{p_s^*}}} \varphi_\rho dx\\
&+\left(\frac{(\theta-1)}{\theta p}-\frac{(\theta-1)}{\theta pK}\right)\int_{{\mathbb{R}^N}\setminus{\Lambda _\varepsilon }} {V(\varepsilon x)|{u_n}|^p} dx\\
\geq& {\frac{(\theta-1)a}{\theta p}}\mu_i+\left(\frac{{1}}{{\theta p}}-\frac{1}{{p_s^*}}\right)\nu_i\\
\geq&{\frac{(\theta-1)a}{\theta p}}{S}{\nu _i}^{\frac{p}{{p_s^*}}}+\left(\frac{{1}}{{\theta p}}-\frac{1}{{p_s^*}}\right)\left(a{S}{\nu _i}^{\frac{p}{{p_s^*}}} + bS^\theta{\nu _i}^{\frac{{\theta p}}{{p_s^*}}}\right)\\
=& \frac{(\theta-1)a}{\theta p}{S} ({\nu '_i})^{p}+\left(\frac{{1}}{{\theta p}}-\frac{1}{{p_s^*}}\right)\left(a{S} ({\nu '_i})^{p}+ bS^2({\nu '_i})^{\theta p}\right)\\
\geq&\frac{(\theta-1)a}{\theta p}{S} {\hat T}^{p}+\left(\frac{{1}}{{\theta p}}-\frac{1}{{p_s^*}}\right)\left(a{S} {\hat T}^{p}+ bS^\theta{\hat T}^{\theta p}\right)\\
=& \frac{a}{p}{S} {\hat T}^{p}+ \frac{b}{\theta p}S^\theta{\hat T}^{\theta p}-\frac{1}{p_s^*} {\hat T}^{p_s^*},
\end{align*}
contracting with the assumption $c<c_0$. Then $u_n\to u$ in $L^{p_s^*}(\Lambda_\varepsilon)$. Observe that
\begin{align*}
I_{n,\rho}^4\leq\int_{({{\mathbb{R}^N}\setminus{\Lambda _\varepsilon }})\cap B_{2\rho}(0)}|g(\varepsilon x,u_n)(u_n-u)|+\int_{{\Lambda _\varepsilon }\cap B_{2\rho}(0)}|g(\varepsilon x,u_n)(u_n-u)|.
\end{align*}
The Sobolev embedding theorem implies that $u_n\to u$ in $L_{loc}^{q}({\mathbb{R}^N})$, $1\leq q<p_s^*$. By $(G_3)$-$(ii)$, we get $\int_{{({\mathbb{R}^N}\setminus{\Lambda _\varepsilon }})\cap B_{2\rho}(0)}|g(\varepsilon x,u_n)(u_n-u)|\to 0$, as $n\to \infty$. Moreover, from $u_n\to u$ in $L^{p_s^*}(\Lambda_\varepsilon)$ we obtain
\begin{align*}
\int_{{\Lambda _\varepsilon }\cap B_{2\rho}(0)}|g(\varepsilon x,u_n)(u_n-u)|\to 0,\ \ \mbox{as}\ \ n\to \infty.
\end{align*}
Then,
\begin{align*}
\mathop {\lim }\limits_{\rho  \to \infty } [\mathop {\lim \sup }\limits_{n \to \infty } |I_{n,\rho}^4|]=0.
\end{align*}
We obtain
\begin{align*}
0\leq\mathop {\lim \sup }\limits_{n \to \infty }\int_{B_R(0)}{P_n}(x)dx\leq\mathop {\lim \sup }\limits_{n \to \infty } (|I_{n,\rho}^1|+|I_{n,\rho}^2|+|I_{n,\rho}^3|+|I_{n,\rho}^4|)=0.
\end{align*}
Then, $\mathop {\lim \sup }\limits_{n \to \infty }\int_{B_R(0)}{P_n}(x)dx=0$. Consequently, we obtain
\begin{align}\label{*5}
 &{M_n}\int_{B_R(0)}\int_{\mathbb{R}^N} {\frac{\left[|u_n(x)-{u_n}(y)|^{p-2}(u_n(x)-{u_n}(y))-|u(x)-u(y)|^{p-2}(u(x)-u(y))\right]}{{|x - y{|^{N + sp}}}}}\nonumber\\
 &\times{[({u_n}(x) - u(x)) - ({u_n}(y) - u(y))]}dxdy +\int_{B_R(0)} V(\varepsilon x)\left(|u_n|^{p-2}u_n-|u|^{p-2}u\right)(u_n-u) dx\rightarrow 0,
\end{align}
as $n\to\infty$.
By Lemma \ref{Simon}, for $p\geq 2$, we know that
\begin{align*}
&\int_{B_R(0)}\int_{\mathbb{R}^N} \frac{\left|[u_n(x)-{u_n}(y)]-[u(x)-u(y)]\right|^p}{|x - y|^{N + sp}}dxdy\\
\leq& c_p \int_{B_R(0)}\int_{\mathbb{R}^N} {\frac{\left[|u_n(x)-{u_n}(y)|^{p-2}(u_n(x)-{u_n}(y))-|u(x)-u(y)|^{p-2}(u(x)-u(y))\right]}{{|x - y{|^{N + sp}}}}}\\
 &\times{[({u_n}(x) - u(x)) - ({u_n}(y) - u(y))]}dxdy\to 0.
\end{align*}
In a similar way, by (\ref{*5}), we get
\begin{align*}
\int_{B_R(0)} V(\varepsilon x)|u_n-u|^p dx\leq&c_p\int_{B_R(0)} V(\varepsilon x)\left(|u_n|^{p-2}u_n-|u|^{p-2}u\right)(u_n-u) dx.
\end{align*}
Now, we consider the case  $1<p<2$. Since $u_n\rightharpoonup u$ in ${\mathcal{W}}_\varepsilon$, there exists $\kappa>0$ such that $\|u_n\|_\varepsilon\leq\kappa$ for all $n\in\mathbb{N}$. Hence, by (\ref{*5}), Simon inequality and H\"{o}lder inequality, we can see that
\begin{align*}
&\int_{B_R(0)}\int_{\mathbb{R}^N} \frac{\left|[u_n(x)-{u_n}(y)]-[u(x)-u(y)]\right|^p}{|x - y|^{N + sp}}dxdy\\
\leq& C_p\int_{B_R(0)}\int_{\mathbb{R}^N} \frac{\left[|u_n(x)-{u_n}(y)|^{p-2}(u_n(x)-{u_n}(y))-|u(x)-u(y)|^{p-2}(u(x)-u(y))\right]^{\frac{p}{2}}}{|x - y|^{N + sp}}\\
&\times\left[(u_n(x)-{u_n}(y)-(u(x)-u(y)))\right]^{\frac{p}{2}}\left(|u_n(x)-{u_n}(y)|^p+|u(x)-u(y)|^p\right)^{\frac{2-p}{2}}dxdy\\
\leq& C_p (\int_{B_R(0)}\int_{\mathbb{R}^N} {\frac{\left[|u_n(x)-{u_n}(y)|^{p-2}(u_n(x)-{u_n}(y))-|u(x)-u(y)|^{p-2}(u(x)-u(y))\right]}{{|x - y{|^{N + sp}}}}}\\
&\times{[({u_n}(x) - u(x)) - ({u_n}(y) - u(y))]}dxdy)^{\frac{p}{2}}\left([u_n]_{s,p}^p+[u]_{s,p}^p\right)^{\frac{2-p}{2}}\\
\leq&C_p (\int_{B_R(0)}\int_{\mathbb{R}^N} {\frac{\left[|u_n(x)-{u_n}(y)|^{p-2}(u_n(x)-{u_n}(y))-|u(x)-u(y)|^{p-2}(u(x)-u(y))\right]}{{|x - y{|^{N + sp}}}}}\\
&\times{[({u_n}(x) - u(x)) - ({u_n}(y) - u(y))]}dxdy)^{\frac{p}{2}}\left([u_n]_{s,p}^{\frac{p(2-p)}{2}}+[u]_{s,p}^{\frac{p(2-p)}{2}}\right)\\
\leq&C'_p (\int_{B_R(0)}\int_{\mathbb{R}^N} {\frac{\left[|u_n(x)-{u_n}(y)|^{p-2}(u_n(x)-{u_n}(y))-|u(x)-u(y)|^{p-2}(u(x)-u(y))\right]}{{|x - y{|^{N + sp}}}}}\\
&\times{[({u_n}(x) - u(x)) - ({u_n}(y) - u(y))]}dxdy)^{\frac{p}{2}}\to 0.
\end{align*}
Then, we can also obtain
\begin{align*}
\int_{B_R(0)} V(\varepsilon x)|u_n-u|^p dx\leq&C''_p\left(\int_{B_R(0)} V(\varepsilon x)\left(|u_n|^{p-2}u_n-|u|^{p-2}u\right)(u_n-u) dx\right)^{\frac{p}{2}}\to0.
\end{align*}
Consequently,
\begin{align*}
&\mathop {\lim }\limits_{n \to \infty } \int_{{B_{{R}}}(0)} \left(a\int_{\mathbb{R}^N} {\frac{{|{u_n}(x) - {u_n}(y){|^p}}}{{|x - y{|^{N + sp}}}}}  dy + V(\varepsilon x)u_n^p\right)dx \\
=&\int_{{B_{{R}}}(0)} \left(a\int_{\mathbb{R}^N} {\frac{{|{u}(x) - {u}(y){|^p}}}{{|x - y{|^{N + sp}}}}}  dy + V(\varepsilon x)u^p\right)dx.
\end{align*}
\end{proof}

\begin{lemma}\label{conv2}
The function ${J_\varepsilon }$ verifies the $(PS)_c$ condition in ${\mathcal{W}}_\varepsilon$ at any level $c<c_0$.
\end{lemma}
\begin{proof}
Let $\{u_n\}$ be a $(PS)_c$ sequence for $J_\varepsilon$. By Lemma \ref{bdd} we see that $\{u_n\}$ is bounded in ${\mathcal{W}}_\varepsilon$. Passing to a subsequence, we have $u_n\rightharpoonup u$ for some $u\in{\mathcal{W}}_\varepsilon$. From Lemma \ref{conv}, it follows that for $\zeta>0$ given, there exists $R_1=R(\zeta)>C/\zeta$ with $C>0$ independent on $\zeta$ such that
\begin{equation*}
\mathop {\lim \sup }\limits_{n \to \infty } \int_{\mathbb{R}^N\setminus{B_{{R_1}}}(0)} \left(a\int_{\mathbb{R}^N} {\frac{{|{u_n}(x) - {u_n}(y){|^p}}}{{|x - y{|^{N + sp}}}}}  dy + V(\varepsilon x)u_n^p\right)dx <\zeta.
\end{equation*}
Therefore, from Lemma \ref{conv1}, one has
\begin{align*}
\|u\|_\varepsilon^p\leq&\mathop {\lim \inf }\limits_{n \to \infty }\|u_n\|_\varepsilon^p\leq\mathop {\lim \sup }\limits_{n \to \infty }\|u_n\|_\varepsilon^p\\
=&\mathop {\lim \sup }\limits_{n \to \infty }\left\{\int_{{B_{{R_1}}}(0)} \left(a\int_{\mathbb{R}^N} {\frac{{|{u_n}(x) - {u_n}(y){|^p}}}{{|x - y{|^{N + sp}}}}}  dy + V(\varepsilon x)u_n^p\right)dx\right.\\
&+\left.\int_{\mathbb{R}^N\setminus{B_{{R_1}}}(0)} \left(a\int_{\mathbb{R}^N} {\frac{{|{u_n}(x) - {u_n}(y){|^p}}}{{|x - y{|^{N + sp}}}}}  dy + V(\varepsilon
 x)u_n^p\right)dx\right\}\\
=&\int_{{B_{{R_1}}}(0)} \left(a\int_{\mathbb{R}^N} {\frac{{|{u_n}(x) - {u_n}(y){|^p}}}{{|x - y{|^{N + sp}}}}}  dy + V(\varepsilon x)u_n^p\right)dx\\
&+\mathop {\lim \sup }\limits_{n \to \infty }\int_{\mathbb{R}^N\setminus{B_{{R_1}}}(0)} \left(a\int_{\mathbb{R}^N} {\frac{{|{u_n}(x) - {u_n}(y){|^p}}}{{|x - y{|^{N + sp}}}}}  dy + V(\varepsilon x)u_n^p\right)dx\\
<&\int_{{B_{{R_1}}}(0)} \left(a\int_{\mathbb{R}^N} {\frac{{|{u_n}(x) - {u_n}(y){|^p}}}{{|x - y{|^{N + sp}}}}}  dy + V(\varepsilon x)u_n^p\right)dx+\zeta,
\end{align*}
where $R_1=R(\zeta)>C/\zeta$. Passing to the limit as $\zeta\to 0$ we have $R_1\to \infty$, which implies that
\begin{align*}
\|u\|_\varepsilon^p\leq&\mathop {\lim \inf }\limits_{n \to \infty }\|u_n\|_\varepsilon^p\leq\mathop {\lim \sup }\limits_{n \to \infty }\|u_n\|_\varepsilon^p\leq\|u\|_\varepsilon^p.
\end{align*}
Consequently, $\|u_n\|_\varepsilon\to\|u\|_\varepsilon$ and so $u_n\to u$ in ${\mathcal{W}}_\varepsilon$.
\end{proof}

\begin{proposition}\label{conv*}
The function ${J_\varepsilon }$ restricted to $\mathcal{N}_\varepsilon$ satisfies the $(PS)_c$ condition in ${\mathcal{W}}_\varepsilon$ at any level $c<c_0$.
\end{proposition}
\begin{proof}
Let $\{u_n\}\subset\mathcal{N}_\varepsilon$ be such that $J_\varepsilon(u_n)\to c$ and $\|J'_\varepsilon(u_n)_{|_{\mathcal{N}_\varepsilon}}\|_*=o_n(1)$. Then there exists $\{\lambda_n\}\subset\mathbb{R}$ such that
\begin{align*}
J'_\varepsilon(u_n)=\lambda_nG'_\varepsilon(u_n)+o_n(1),
\end{align*}
where $J'_\varepsilon(u)=\left\langle {{J'_\varepsilon }({u}),{u}} \right\rangle$. Moreover, since $u_n\in \mathcal{N}_\varepsilon$, we know that
\begin{align}\label{0}
0=\left\langle {{J'_\varepsilon }({u_n}),{u_n}} \right\rangle=\lambda_n\left\langle {{G'_\varepsilon }({u_n}),{u_n}} \right\rangle+o_n(1)\|u_n\|_{{\mathcal{W}}_\varepsilon}.
\end{align}
Next, we will show that $\lambda_n\to 0$ as $n\to\infty$. Indeed, by $(G_4)$ and $\left\langle {{J'_\varepsilon }({u_n}),{u_n}} \right\rangle=o_n(1)$, we deduce that
\begin{align*}
\left\langle {{G'_\varepsilon }({u_n}),{u_n}} \right\rangle
=&p\|u_n\|_\varepsilon^p+\theta pb[u_n]_{s,p}^{\theta p}-\int_{\mathbb{R}^N}g'(\varepsilon x,u_n)u_n^2dx-\int_{\mathbb{R}^N}g(\varepsilon x,u_n)u_ndx\\
=&-(\theta-1)p\|u_n\|_\varepsilon^p+\int_{\mathbb{R}^N}[(\theta p-1)g(\varepsilon x,u_n)u_n-g'(\varepsilon x,u_n)u_n^2]dx\\
\leq&-(\theta-1)p\|u_n\|_\varepsilon^p+(\theta-1)p\int_{\mathbb{R}^N}(1-\chi)(\varepsilon x)\tilde f({u_n})u_ndx\\
\leq&-((\theta-1)p-\frac{(\theta-1)p}{K})\|u_n\|_\varepsilon^p.
\end{align*}
Thus, we may assume that $\left\langle {{G'_\varepsilon }({u_n}),{u_n}} \right\rangle\to l<0$. It follows from (\ref{0}) that $\lambda_n\to 0$ as $n\to\infty$ and then we see that ${J'_\varepsilon }({u_n})\to 0$ as $n\to\infty$. Hence, $\{u_n\}$ is a $(PS)_c$ sequence for $J_\varepsilon$.
\end{proof}

\section{the autonomous problem}
In this section, we are concerned with the autonomous equation:
\begin{eqnarray}\label{JV0}
(a + b\left[ u \right]_{s,p}^{(\theta-1)p})(-\Delta)_p^su + V_0{u^{p - 1}} = f(u) + {u^{p_s^* - 1}}\ \ & \mbox{in}\ {\mathbb{R}^N},\ \ \ u>0,\ u\in{W^{s,p}}(\mathbb{R}^N),
\end{eqnarray}
where $V_0$ is given in $(V_2)$.\

The functional of (\ref{JV0}) is denoted by
\begin{equation*}
{J_0 }(u) = \frac{1}{p}\|u\|_0^p + \frac{b}{{\theta p}}[u]_{_{s,p}}^{\theta p}  - \int_{\mathbb{R}^N} {F(u)} dx- \frac{1}{p_s^*}\int_{\mathbb{R}^N} {(u^+)^{p_s^*}} dx.
\end{equation*}
The Nehari manifold corresponding to (\ref{JV0}) is defined by
\begin{equation*}
{\mathcal{N}}_0:=\left\{u\in{{\mathcal{W}}_0\setminus\{0\}}:\left\langle {{J'_0 }(u),u} \right\rangle =0\right\},
\end{equation*}
where ${\mathcal{W}}_0={W^{s,p}}(\mathbb{R}^N)$ is endowed with the norm $\|u\|_0^p=a[u]_{s,p}^p +\int_{\mathbb{R}^N} {V_0} |u{|^p}dx$.\

As (\ref{c_e}), we have the following characterization of infimum of $J_0$ on $\mathcal{N}_0$
\begin{equation}\label{c_V0}
{c_{V_0} } = \mathop {\inf }\limits_{u \in {\mathcal{W}_0 }\setminus\{0\}} \mathop {\sup }\limits_{t \ge 0} {J_0 }(tu) = \mathop {\inf }\limits_{u \in \mathcal{N}_0} {J_0}(u)>0,
\end{equation}
and as Lemma \ref{compare}, we can show that $c_{V_0}<c_0$.
\begin{lemma}\label{pgs}
Let $\{u_n\}\subset\mathcal{W}_0$ be a $(PS)_{c_{V_0}}$ sequence of ${J_0 }$. Then there exists $u\in\mathcal{W}_0\setminus\{0\}$ such that $u_n\to u$ in $\mathcal{W}_0$. Moreover, $u$ is a positive ground state of the equation (\ref{JV0}).
\end{lemma}
\begin{proof}
It is easy to check that $J_0$ has a mountain pass geometry, so there exists a sequence $\{u_n\}\subset\mathcal{W}_0$ such that $J_0(u_n)\to c_{V_0}$ and $J'_0(u_n)\to 0$. Thus, we can easily see that $\{u_n\}$ is bounded in $\mathcal{W}_0$ and we may assume that $u_n\rightharpoonup u$ in $\mathcal{W}_0$ and $[u_n]_{s,p}^{(\theta-1)p}\rightarrow B^{(\theta-1)p}$. Similar argument to the proof of Lemma \ref{nonv}, we can deduce that the sequence $\{u_n\}$ is nonvanishing, $i.e.$, there exists a sequence $\{x_n\}\subset \mathbb{R}^N$ and $R_0>0$, $\beta_0>0$ such that
\begin{equation*}
\int_{B_R(x_n)} |u_n|^p\ dx\geq\beta_0.
\end{equation*}
Then we suppose that $u\neq 0$. Since $\left\langle {{J'_0 }({u}),\varphi} \right\rangle=o_n(1)$, we can deduce that for all $\varphi\in \mathcal{C}_0^\infty(\mathbb{R}^N)$,
\begin{align}\label{aa}
&a\iint_{\mathbb{R}^{2N}}\frac{{|u(x) - u(y){|^{p - 2}}(u(x) - u(y))(\varphi (x) - \varphi (y))}}{{|x - y{|^{N + sp}}}}dxdy+\int_{{\mathbb{R}^N}}V(\varepsilon x)u^{p-1}\varphi dx\nonumber\\
&+bB^{(\theta-1)p}\iint_{{\mathbb{R}^{2N}}}\frac{{|u(x) - u(y){|^{p - 2}}(u(x) - u(y))(\varphi (x) - \varphi (y))}}{{|x - y{|^{N + sp}}}}dxdy\nonumber\\
=&\int_{{\mathbb{R}^{N}}}f(u)u\varphi+\int_{{\mathbb{R}^N}}(u^+)^{p_s^*-1}\varphi
\end{align}
where $B^{(\theta-1)p}:=\mathop {\lim}\limits_{n\to \infty}[u_n]_{s,p}^{(\theta-1)p}$. Let us note that $B^{(\theta-1)p}\geq [u]_{s,p}^{(\theta-1)p}$ by the weakly semi-continuous of norm. If by construction $B^{(\theta-1)p}>[u]_{s,p}^{(\theta-1)p}$, we may use (\ref{aa}) to deduce that $\left\langle {{J'_0 }({u}),u} \right\rangle<0$. Moreover, conditions $(F_1)$ and $(F_2)$ imply that $\left\langle {{J'_0 }({\tau u}),\tau u} \right\rangle=0$. From $t_0\in(\tau,1)$, $(F_3)$ and $(F_4)$ we get
\begin{align}\label{aaa}
c_{V_0}\leq& J_0(t_0u)-\frac{1}{\theta p}\left\langle {{J'_0 }({t_0 u}),t_0 u} \right\rangle<J_0(u)-\frac{1}{\theta p}\left\langle {{J'_0 }({ u}),u} \right\rangle\nonumber\\
\leq & \mathop {\lim \inf }\limits_{n \to \infty }[ J_0(u_n)-\frac{1}{\theta p}\left\langle {{J'_0 }({ u_n}), u_n} \right\rangle]=c_{V_0},
\end{align}
which gives a contradiction. Therefore, $B^{(\theta-1)p}=[u]_{s,p}^{(\theta-1)p}$ and we deduce that $J'_0(u)=0$. Using the fact that $\left\langle {{J'_0 }({ u}), u^-} \right\rangle=0$, where $u^-=\min\{u,0\}$, by $f(t)=0$ for $t\leq0$, and $|x - y{|^{p - 2}}(x - y)({x^ - } - {y^ - }) \ge |{x^ - } - {y^ - }{|^p}$, where $x^-=\min\{x,0\}$, we can deduce that
\begin{align*}
\|u^-\|_0^p\leq&a\iint_{\mathbb{R}^{2N}}\frac{{|u(x) - u(y){|^{p - 2}}(u(x) - u(y))(u^- (x) - u^- (y))}}{{|x - y{|^{N + sp}}}}dxdy+\int_{{\mathbb{R}^N}}V(\varepsilon x)|u|^{p-2}uu^- dx\\
\leq&\int_{{\mathbb{R}^N}}f(u)u^-+\int_{{\mathbb{R}^N}}(u^+)^{p_s^*-1}u^-=0,
\end{align*}
which implies that $u^-=0$, that is $u\geq0$. Moreover, proceeding as in the proof of Lemma \ref{estim} below, we can see that $u\in L^\infty({{\mathbb{R}^N}})$. By Corollary $5.5$ in \cite{ASM}, we deduce that $u\in\mathcal{C}^{0,\alpha}({{\mathbb{R}^N}})$, and applying maximum principle \cite{PMQ} we can conclude that $u>0$ in $\mathbb{R}^N$. Finally, arguing as in (\ref{aaa}) with $t_0=1$, we can show that $u$ is a positive ground state solution to (\ref{JV0}).
\end{proof}
\begin{lemma}
The following limits holds
\begin{equation*}
\mathop {\lim \sup }\limits_{\varepsilon  \to 0} {c_\varepsilon } \le {c_{{V_0}}}.
\end{equation*}
Moreover, $\mathop {\lim  }\limits_{\varepsilon  \to 0} {c_\varepsilon } = {c_{{V_0}}}$.
\end{lemma}
\begin{proof}
For any $\varepsilon>0$ we set $\omega_\varepsilon(x):=\psi_\varepsilon(x)\omega(x)$, where $\omega$ is a positive ground state solution, and $\psi_\varepsilon(x):=\psi(\varepsilon x)$ with $\psi\in \mathcal{C}_0^\infty(\mathbb{R}^N)$, $\psi\in [0,1]$, $\psi(x)=1$ if $|x|\leq\frac{1}{2}$ and $\psi(x)=0$ if $|x|\geq1$. Here we assume that $supp(\psi)\subset B_1\subset\Lambda$. Using Lemma \ref{Lion} and the dominated convergence theorem we can see that $\omega_\varepsilon\to \omega$ in $\mathcal{W}_0$ and $J_0(\omega_\varepsilon)\to J_0(\omega)=c_{V_0}$ as $\varepsilon\to 0$. For each $\varepsilon>0$ there exists $t_\varepsilon>0$ such that
\begin{equation*}
J_\varepsilon(t_\varepsilon\omega_\varepsilon)=\mathop {\max }\limits_{t\geq0}J_\varepsilon(t\omega_\varepsilon).
\end{equation*}
Then, $J'_\varepsilon(t_\varepsilon\omega_\varepsilon)=0$ and this implies that
\begin{align}\label{t1}
&\frac{1}{t_\varepsilon^{(\theta-1)p}}\left(a[\omega_\varepsilon]_{s,p}^p+\int_{\mathbb{R}^N}V(\varepsilon x)\omega_\varepsilon^p dx\right)+b[\omega_\varepsilon]_{s,p}^{\theta p}\nonumber\\
=&\int_{{\mathbb{R}^N}}\frac{f(t_\varepsilon\omega_\varepsilon)}{(t_\varepsilon\omega_\varepsilon)^{\theta p-1}}{\omega_\varepsilon^{\theta p}}dx
+t_\varepsilon^{p_s^*-\theta p}\int_{\mathbb{R}^N}|\omega_\varepsilon|^{p_s^*}dx.
\end{align}
By $(F_1)-(F_4)$, $\omega\in \mathcal{W}_0$ and (\ref{t1}) it follows that $t_\varepsilon\to 1$ as $\varepsilon\to 0$. On the other hand,
\begin{align*}
c_\varepsilon\leq\mathop {\max }\limits_{t\geq0}J_\varepsilon(t\omega_\varepsilon)=J_\varepsilon(t_\varepsilon\omega_\varepsilon)=J_0(t_\varepsilon\omega_\varepsilon)
+\frac{t_\varepsilon^p}{p}\int_{{\mathbb{R}^N}}(V(\varepsilon x)-V_0)\omega_\varepsilon^p dx.
\end{align*}
Since $V(\varepsilon x)$ is bounded on the support of $\omega_\varepsilon$, by the dominated convergence theorem and the above inequality, we obtain the thesis.\

Moreover, by (\ref{c_e}), (\ref{c_V0}) and $J_0(u)\leq J_\varepsilon(u)$ for any $u\in \mathcal{W}_0$, we have $c_{V_0}\leq c_\varepsilon$. Then $c_{V_0}\leq \mathop {\lim \inf }\limits_{\varepsilon  \to 0} {c_\varepsilon }$. Therefore, $\mathop {\lim  }\limits_{\varepsilon  \to 0} {c_\varepsilon } = {c_{{V_0}}}$.
\end{proof}

\section{Multiple solutions for the Modified problem}
We consider $\delta> 0$ such that $M_\delta\subset \Lambda$. Let $w$ be a positive ground state solution of problem (\ref{JV0}), and choose $\psi\in \mathcal{C}_c^\infty(\mathbb{R}^+,[0,1])$ satisfying $\psi(s)=1$ if $0\leq s\leq\frac{\delta}{2}$, and $\psi(s)=0$  if $s\geq\delta$.\

For each $y\in M=\left\{x\in\Lambda: V(x)=V_0\right\}$, we define
\begin{equation*}
{\Upsilon _{\varepsilon ,y}}(x) = \psi (|\varepsilon x - y|)w\left(\frac{{\varepsilon x - y}}{\varepsilon }\right).
\end{equation*}
Then there exists $t_\varepsilon>0$ such that $\mathop {\max }\limits_{t \ge 0} {J_\varepsilon }(t{\Upsilon _{\varepsilon ,y}}) = {J_\varepsilon }({t_\varepsilon }{\Upsilon _{\varepsilon ,y}})$. Define $\Phi_\varepsilon: M\rightarrow\mathcal{N}_\varepsilon$ by $\Phi_\varepsilon(y)=t_\varepsilon\Upsilon_{\varepsilon,y}$.
\begin{lemma}\label{yM}
The following limits hold:
$\mathop {\lim }\limits_{\varepsilon  \to 0} {J_\varepsilon(\Phi_\varepsilon(y)) } = {c_{{V_0}}}$ uniformly in $y\in M$.
\end{lemma}
\begin{proof}
Suppose by condition that there exist $\delta_0$, $y_n\in M$ and $\varepsilon_n\to 0$ such that
\begin{align}\label{35}
\left|J_{\varepsilon_n}(\Phi_{\varepsilon_n}(y_n))-c_{V_0}\right|\geq \delta_0.
\end{align}
Note that for each $n\in\mathbb{N}$ and for all $z\in B_{\frac{\delta}{\varepsilon_n}}(0)$, we have $\varepsilon_n z\in B_\delta(0)$. So
\begin{equation*}
\varepsilon_n z+y_n\in B_\delta(0)\subset M_\delta\subset\Lambda.
\end{equation*}
Using the definition of $\Phi_{\varepsilon_n}$ and making the change of the variable $z=x-\frac{y_n}{\varepsilon_n}$, we have
\begin{align}\label{36}
A_n:=&\frac{1}{{t_{{\varepsilon _n}}^{(\theta-1)p}}}\|\Upsilon_{\varepsilon_n,y_n}\|_{\varepsilon_n}^p+b[\Upsilon_{\varepsilon_n,y_n}]_{s,p}^{\theta p}\nonumber\\
=&\frac{1}{{t_{{\varepsilon _n}}^{\theta p - 1}}}\int_{\mathbb{R}^N}f({t_{{\varepsilon _n}}}\psi (|{\varepsilon _n}z|)w(z))\psi (|{\varepsilon _n}z|)w(z)dz+t_{{\varepsilon _n}}^{p_s^* - \theta p}\int_{\mathbb{R}^N}[\psi (|{\varepsilon _n}z|)w(z)]^{p_s^*}dz:=B_n.
\end{align}
By using Lemma \ref{Lion1}, we have
\begin{align}\label{37}
\|\Upsilon_{\varepsilon_n,y_n}\|_{\varepsilon_n}^p\rightarrow \|w\|_0^p.
\end{align}
And by Lebesgue's theorem, we can verify that
\begin{align}\label{38}
\int_{\mathbb{R}^N} f(\Upsilon_{\varepsilon_n,y_n})\Upsilon_{\varepsilon_n,y_n}\rightarrow \int_{\mathbb{R}^N} f(w)w,\ \ \int_{\mathbb{R}^N} |\Upsilon_{\varepsilon_n,y_n}|^{p_s^*}\rightarrow \int_{\mathbb{R}^N} |w|^{p_s^*}.
\end{align}
If $t_{\varepsilon_n}\rightarrow\infty$, then we get from (\ref{36}) that
\begin{align}\label{39}
B_n\geq\int_{B_{\delta/2}(0)}t_{{\varepsilon _n}}^{p_s^* - \theta p}[\psi (|{\varepsilon _n}z|)w(z)]^{p_s^*}\geq t_{{\varepsilon _n}}^{p_s^* - \theta p}k^{p_s^*}|{B_{\delta/2}(0)}|\rightarrow\infty
\end{align}
as $\varepsilon_n\rightarrow 0$, where $k:=\mathop {\inf }\limits_{x \in {B_{\delta /2}}(0)} w(x)$. However, $A_n$ in (\ref{36}) is bounded. This yields a contradiction. Hence, $\{t_{\varepsilon_n}\}$ is bounded and assume that $t_{\varepsilon_n}\rightarrow t_0\geq0$. Since $t_{\varepsilon_n}\Upsilon_{\varepsilon_n,y_n}\in \mathcal{N}_{\varepsilon_n}$, we can easily know $t_{\varepsilon_n}\|\Upsilon_{\varepsilon_n,y_n}\|_{{\varepsilon_n}}\geq \rho>0$. With the use of (\ref{37}), we have $t_0>0$. Since $t_{\varepsilon_n}\to t_0$, passing limit in (\ref{36}) we have
\begin{align*}
\frac{1}{{t_0^p}}\|w\|_0^{p}+b[w]_{s,p}^{\theta p}=\int_{\mathbb{R}^N}\frac{f(t_0w)}{{(t_0w)^{\theta p - 1}}}w^{\theta p}dz+t_0^{p_s^* - \theta p}\int_{\mathbb{R}^N}|w|^{p_s^*}dz.
\end{align*}
Observing that $w$ is a positive ground state solution of (\ref{JV0}), then by $(F_4)$ we know that $t_0=1$. Letting $n\to\infty$ and recalling $t_{\varepsilon_n}\rightarrow 1$, using (\ref{37}) and (\ref{38}) we get
\begin{align*}
\mathop {\lim }\limits_{n \to \infty } {J_{{\varepsilon _n}}}({\Phi _{{\varepsilon _n}}}) = {J_0}(w) = {c_{{V_0}}},
\end{align*}
which contradicts with (\ref{35}). The ends the proof.
\end{proof}

For any $\delta>0$, let $\rho=\rho(\delta)>0$ be such that $M_\delta\subset B_\rho(0)$. Let $\chi: \mathbb{R}^N\rightarrow\mathbb{R}^N$. Define $\chi(x)=x$ if $|x|<\rho$ and $\chi(x)=\frac{\rho x}{|x|}$ if $|x|\geq\rho$. Finally, we define $\beta_\varepsilon:\mathcal{N}_\varepsilon\rightarrow\mathbb{R}^N$ by
\begin{equation*}
{\beta_\varepsilon}(u):=\frac{\int_{\mathbb{R}^N}\chi(\varepsilon x)|u(x)|^pdx}{\int_{\mathbb{R}^N}|u(x)|^pdx}.
\end{equation*}
Since $M\subset B_\rho(0)$, by the definition of $\chi$ and Lebesgue's theorem, arguing as in Lemma $3.13$ in \cite{AI}, we conclude that
\begin{equation}\label{4.11}
\mathop {\lim }\limits_{\varepsilon  \to 0} {\beta_\varepsilon(\Phi_\varepsilon(y)) } = y\ \ \ \mbox{uniformly}\ \mbox{for}\ \ y\in M.
\end{equation}

\begin{lemma}\label{conc}
Let $\varepsilon_n\rightarrow 0$ and $\{u_n\}\subset \mathcal{N}_{\varepsilon_n}$ be such that $J_{\varepsilon_n}(u_n)\rightarrow c_{V_0}$. Then there exists a sequence $\{y_n\}\subset \mathbb{R}^N$ such that $v_n(x)=u_n(x+y_n)$ has a convergent subsequence in $\mathcal{W}_0$. Moreover, up to a subsequence, ${{\tilde y}_n}: = {\varepsilon _n}{y_n}\to y_0\in M$.
\end{lemma}
\begin{proof}
Arguing as in the proof of Lemma \ref{bdd}, $\{u_n\}$ is bounded in $\mathcal{W}_0$ and non-vanishing. Then, there exist $R, \delta'>0$ and $y_n\in \mathbb{R}^N$ such that $\int_{B_R(y_n)}u_n^pdx>\delta'$. Define $v_n(x)=u_n(x+y_n)$, then passing to a subsequence, we assume $v_n\rightharpoonup v\neq 0$ in $\mathcal{W}_0$. Let $t_n>0$ be such that ${\tilde v_n} = {t_n}{v_n}\in \mathcal{N}_0$. Set ${{\tilde y}_n}: = {\varepsilon _n}{y_n}$. Then
\begin{equation*}
c_{V_0}\leq J_0({\tilde v_n})=J_0({t_n}{v_n})=J_0({t_n}{u_n})\leq J_{\varepsilon_n}({t_n}{u_n})\leq J_{\varepsilon_n}({u_n})=c_{V_0}+o_n(1).
\end{equation*}
Then $J_0({\tilde v_n})\rightarrow c_{V_0}$. Moreover, $\{\tilde v_n\}$ is bounded in $\mathcal{W}_0$ and $\tilde v_n\rightharpoonup\tilde v$. We assume that $t_n\to t^*>0$. From the uniqueness of the weak limit we have that $\tilde v=t^*v\not\equiv 0$. This combined with Lemma \ref{pgs} implies that $\tilde v_n\rightarrow\tilde v$ in $\mathcal{W}_0$. Consequently, $v_n\to v$ in $\mathcal{W}_0$ as $n\rightarrow\infty$.\

Next, we claim that $\{{\tilde y}_n\}$ is bounded. Indeed, suppose by contradiction that $|{\tilde y}_n|\to\infty$. Choose $R>0$ such that $\Lambda\subset B_R(0)$, then for $n$ large enough, we have $|{\tilde y}_n|>2R$, and for each $x\in B_{\frac{R}{\varepsilon_n}}(0)$, we have
\begin{equation*}
|\varepsilon_n x+{\tilde y}_n|>|{\tilde y}_n|-|\varepsilon_n x|\geq 2R-R=R.
\end{equation*}
Hence, using $v_n\rightarrow v$ in $\mathcal{W}_0$ and Lebesgue's theorem, we obtain
\begin{align*}
a[v_n]_{s,p}^p+\int_{{\mathbb{R}^N}}V_0 v_n^p dx\leq& a[v_n]_{s,p}^p+\int_{{\mathbb{R}^N}}V(\varepsilon_n x+{\tilde y}_n) v_n^p dx+b[v_n]_{s,p}^{\theta p}\\
=&\int_{B_{\frac{R}{\varepsilon_n}}(0)}g(\varepsilon_n x+{\tilde y}_n,v_n)v_n dx+\int_{{\mathbb{R}^N}\setminus {B_{\frac{R}{\varepsilon_n}}(0)}}g(\varepsilon_n x+{\tilde y}_n,v_n)v_n dx\\
\leq&\frac{V_0}{K}\int_{\mathbb{R}^N} v_n^p dx+\int_{{\mathbb{R}^N}\setminus {B_{\frac{R}{\varepsilon_n}}(0)}}\left(f(v_n)v_n+v_n^{p_s^*}\right)dx\\
\leq& \frac{V_0}{K}\int_{\mathbb{R}^N} v_n^p dx+o_n(1).
\end{align*}
Then $v_n\to 0$ in $\mathcal{W}_0$, which contradicts to $v\not\equiv0$. Therefore, $\{{\tilde y}_n\}$ is bounded and we may assume that ${\tilde y}_n\to y_0\in \mathbb{R}^N$. If $y_0\notin {\bar \Lambda }$, as before we can infer that $v_n\to 0$ in $\mathcal{W}_0$. This is impossible. Thus we have that $y_0\in {\bar \Lambda }$.\

In order to prove that $V(y_0)=V_0$, we can suppose by contradiction that $V_0<V(y_0)$. Then from $\tilde v_n\rightarrow\tilde v$ in $\mathcal{W}_0$ and the invariance of $\mathbb{R}^N$ by translation, to obtain
\begin{align*}
c_{V_0}=&J_0(\tilde v)\\
<&\mathop {\lim\inf }\limits_{n  \to \infty} \left\{\frac{a}{p}[\tilde v_n]_{s,p}^p+\frac{1}{p}\int_{{\mathbb{R}^N}}V(\varepsilon_n (x+{ y}_n)) |\tilde v_n|^pdx+\frac{b}{\theta p}[\tilde v_n]_{s,p}^{\theta p}-\int_{{\mathbb{R}^N}}F(\tilde v_n)dx-\frac{1}{p_s^*}|\tilde v_n^+|^{p_s^*}dx\right\}\\
\leq&\mathop {\lim\inf }\limits_{n  \to \infty} \left\{\frac{at_n^p}{p}[u_n]_{s,p}^p+\frac{t_n^p}{p}\int_{{\mathbb{R}^N}}V(\varepsilon_n z) |u_n|^pdx+\frac{bt_n^{\theta p}}{\theta p}[u_n]_{s,p}^{\theta p}-\int_{{\mathbb{R}^N}}G(\varepsilon_nz,t_nu_n)dx\right\}\\
=&\mathop {\lim\inf }\limits_{n  \to \infty} J_{\varepsilon_n}(t_nu_n)\leq \mathop {\lim\inf }\limits_{n  \to \infty} J_{\varepsilon_n}(u_n)=c_{V_0},
\end{align*}
which yields a contradiction. Therefore, $V(y_0)=V_0$ and $z_0\in M$. The condition $(V_2)$ shows that $y_0\notin \partial M$, so, $y_0\in M$. The proof is complete.
\end{proof}

Let $h(\varepsilon)$ be any positive function satisfying $h(\varepsilon)\to 0$ as $\varepsilon\to 0$. Define the set
\begin{equation*}
\Sigma_\varepsilon = \left\{u\in \mathcal{N}_\varepsilon: J_\varepsilon(u)\leq c_{V_0}+h(\varepsilon)\right\}.
\end{equation*}
For any $y\in M$, we deduce from Lemma \ref{yM} that $h(\varepsilon)=\left|J_\varepsilon(\Phi_\varepsilon(y)))-c_{V_0}\right|\rightarrow 0$ as $\varepsilon\rightarrow 0$. Thus $\Phi_\varepsilon(y)\in\Sigma_\varepsilon$ and $\Sigma_\varepsilon\neq\emptyset$ for $\varepsilon>0$. Moreover, as in Lemma $3.14$ in \cite{AI}, we have the following lemma
\begin{lemma}\label{bM}
For any $\delta>0$, there holds that $\mathop {\lim }\limits_{\varepsilon  \to 0} \mathop {\sup }\limits_{u \in \Sigma_\varepsilon} dist({\beta _\varepsilon }(u) - {M_\delta }) = 0$.
\end{lemma}
\begin{lemma}\label{4.3}\cite{KCC}
Let $I$ be a $C^1$-functional defined on a $C^1$-Finsler manifold $\nu$. If $I$ is bounded from below and satisfies the $(PS)$ condition, then $I$ has at least $cat_{\nu}(\nu)$ distinct critical points.
\end{lemma}
\begin{lemma}\label{4.4}\cite{BC}
Let $\Gamma$, $\Omega^+$, $\Omega^-$ be closed sets with $\Omega^-\subset\Omega^+$. Let $\Phi:\Omega^-\rightarrow \Gamma$, $\beta:\Gamma\rightarrow\Omega^+$ be two continuous maps such that $\beta\circ\Phi$ is homotopically equivalent to the embedding $Id: \Omega^-\rightarrow\Omega^+$. Then $cat_{\Gamma}(\Gamma)\geq cat_{\Omega^+}(\Omega^-)$.
\end{lemma}

\begin{proposition}
For any $\delta>0$ such that $M_\delta\subset\Lambda$, there exists ${{\bar \varepsilon }_\delta }>0$ such that, for any $\varepsilon\in(0,{{\bar \varepsilon }_\delta })$, problem (\ref{1.1*}) has at least $cat_{M_\delta}(M)$ positive solutions.
\end{proposition}
\begin{proof}
Given $\delta>0$ such that $M_\delta\subset\Lambda$, we can use Lemma \ref{yM}, Lemma \ref{bM} and (\ref{4.11}) to obtain ${{\bar \varepsilon }_\delta }>0$ such that for any $\varepsilon\in(0,{{\bar \varepsilon }_\delta })$, the diagram
\begin{equation*}
 M\xrightarrow{{\Phi _\varepsilon }}{\Sigma _\varepsilon }\xrightarrow{\beta_\varepsilon}{M_\delta }
\end{equation*}
is well defined. In view of (\ref{4.11}), for $\varepsilon$ small enough, we can denote by $\beta_\varepsilon(\Phi_\varepsilon(y))=y+\theta(y)$ for $y\in M$, where $|\theta(y)|<\delta'/2$ uniformly for $y\in M$. Define $S(t,y)=y+(1-t)\theta(y)$. Thus $S:[0,1]\times M\rightarrow M_\delta$ is continuous. Obviously, $S(0,y)=\beta_\varepsilon(\Phi_\varepsilon(y))$ and $S(1,y)=y$ for all $y\in M$. That is, $\beta_\varepsilon \circ {{\Phi _\varepsilon }}$ is homotopically equivalent to $Id: M\rightarrow M_\delta$. By Lemma \ref{4.4}, we obtain that
\begin{equation*}
\mathop {cat}\limits_{\Sigma_\varepsilon} ({\Sigma_\varepsilon})\geq\mathop {cat}\limits_{M_\delta} ({M_\delta}).
\end{equation*}
Since $c_{V_0}\leq c_0$, we can use the definition of $\Sigma_\varepsilon$ and Proposition \ref{conv*} to conclude that $J_\varepsilon$ satisfies $(PS)$ condition in $\Sigma_\varepsilon$ for all small $\varepsilon>0$. Therefore, Lemma \ref{4.3} proves at least $\mathop {cat}\limits_{\Sigma_\varepsilon} ({\Sigma_\varepsilon})$ critical points of $J_\varepsilon$ restricted to ${\Sigma_\varepsilon}$. Using the same arguments as in the proof of Proposition \ref{conv*}, we can conclude that a critical point of the functional $J_\varepsilon$ on $\mathcal{N}_\varepsilon$ is $\mathcal{W}_\varepsilon$ and therefore a week solution for the problem the theorem is proved.
\end{proof}

\section{Proof of the main result}
In the last section, we provide the proof of our main result. Firstly, we develop a iteration scheme \cite{Moser} which will be the main key to deduce that the solutions to (\ref{1.2}) are indeed solutions to (\ref{1.1*}).
\begin{lemma}\label{estim}
Let $\varepsilon_n\to 0$ and $\{u_n\}\in\mathcal{N}_{\varepsilon_n}$ be a solution to (\ref{1.2}). Then $v_n(x)=u_n(\cdot+y_n)$ satisfies $v_n\in L^\infty(\mathbb{R}^N)$ and there exists $C>0$ such that
\begin{equation*}
|v_n|_\infty\leq C\ \ \text{for}\ \text{all}\ n\in\mathbb{N},
\end{equation*}
where $y_n$ was given by Lemma \ref{conc}. Moreover
\begin{equation*}
\mathop {\lim }\limits_{|x| \to \infty } {v_n}(x) = 0\ \ \text{uniformly}\ \text{in}\ n\in\mathbb{N}.
\end{equation*}
\end{lemma}
\begin{proof}
For any $L>0$ and $\beta>1$, let us define the function
\begin{equation*}
\gamma(v_n):=\gamma_{L,\beta}(v_n)=v_nv_{L,n}^{p(\beta-1)}\in\mathcal{W}_\varepsilon,
\end{equation*}
where $v_{L,n}=\min\{L,v_n\}$. Since $\gamma$ is an increasing function, we have
\begin{equation*}
(a-b)(\gamma(a)-\gamma(b))\geq0\ \ \text{for}\ \text{any}\ \ a, b\in\mathbb{R}.
\end{equation*}
Let us introduce the following functions
\begin{equation*}
\mathcal{Q}(t):=\frac{|t|^p}{p},\qquad \text{and}\qquad\Gamma(t):=\int_0^t(\gamma'(\tau))^{\frac{1}{p}}d\tau.
\end{equation*}
Then, applying Jensen's inequality we get for all $a,b\in\mathbb{R}$ such that $a>b$
\begin{align*}
\mathcal{Q}'(a-b)(\gamma(a)-\gamma(b))=&(a-b)^{p-1}(\gamma(a)-\gamma(b))=(a-b)^{p-1}\int_a^b\gamma'(t)dt\\
=&(a-b)^{p-1}\int_a^b(\Gamma'(t))^pdt\geq\left(\int_a^b(\Gamma'(t))dt\right)^p.
\end{align*}
The same argument works when $a\leq b$. Therefore
\begin{align}\label{3.7}
\mathcal{Q}'(a-b)(\gamma(a)-\gamma(b))\geq\left|\Gamma(a)-\Gamma(b)\right|^p\ \ \text{for}\ \text{any}\ \ a, b\in\mathbb{R}.
\end{align}
From (\ref{3.7}), we can see that
\begin{align}\label{3.8}
\left|\Gamma(v_n)(x)-\Gamma(v_n)(y)\right|^p\leq|v_n(x)-v_n(y)|^{p-2}(v_n(x)-v_n(y))\left((v_nv_{L,n}^{p(\beta-1)})(x)-(v_nv_{L,n}^{p(\beta-1)})(y)\right).
\end{align}
Choosing $\gamma(v_n)=v_nv_{L,n}^{p(\beta-1)}$ as a test function in (\ref{1.2}) and using (\ref{3.8}) we get
\begin{align}\label{3.9}
&a[\Gamma(v_n)]_{s,p}^p+\int_{\mathbb{R}^N}V_n(x)|v_n|^pv_{L,n}^{p(\beta-1)}dx\nonumber\\
\leq&(a+b[v_n]_{s,p}^{(\theta-1)p})\iint_{\mathbb{R}^{2N}}\frac{{|{{{v}}_n}(x) - {v_n}(y){|^{p - 2}}({{{v}}_n}(x) - {v_n}(y))}}{{|x - y{|^{N + sp}}}}\left( {({v_n}v_{L,n}^{p(\beta  - 1)})(x) - ({v_n}v_{L,n}^{p(\beta  - 1)})(y)} \right)dxdy\nonumber\\
&+\int_{\mathbb{R}^N}V_n(x)|v_n|^pv_{L,n}^{p(\beta-1)}dx\nonumber\\
=&\int_{\mathbb{R}^N}g_n(v_n)v_nv_{L,n}^{p(\beta-1)}dx,
\end{align}
where we use the notations $V_n(x):=V(\varepsilon_n(x+y_n))$ and $g_n(v_n):=g(\varepsilon_n(x+y_n),v_n)$. Observing that
\begin{align*}
\Gamma(v_n)\geq\frac{1}{\beta}v_nv_{L,n}^{\beta-1}
\end{align*}
and by Lemma \ref{SE}, we have
\begin{align}\label{3.10}
[\Gamma(v_n)]_{s,p}^p\geq C_*^{-1}|\Gamma(v_n)|_{p_s^*}^p\geq \left(\frac{1}{\beta}\right)^pC_*^{-1}\|v_nv_{L,n}^{\beta-1}\|_{p_s^*}^p.
\end{align}
On the other hand, by $(G_1)$ and $(G_2)$, for any $\xi>0$, there exists $C_\xi>0$ such that
\begin{align}\label{3.11}
|g_n(v_n)|\leq \xi|v_n|^{p-1}+C_\xi|v_n|^{p_s^*-1}.
\end{align}
Thus, taking $\xi\in (0,V_1)$, using (\ref{3.10}) and (\ref{3.11}), we can see that (\ref{3.9}) yields
\begin{align}\label{3.12}
\|w_{L,n}\|_{p_s^*}^p\leq C\beta^p\int_{\mathbb{R}^N}|v_n|^{p_s^*}v_{L,n}^{p(\beta-1)} dx,
\end{align}
where $w_{L,n}=v_nv_{L,n}^{\beta-1}$ . Now, we take $\beta=\frac{p_s^*}{p}$ and fix $R>0$. Noting that $0\leq v_{L,n}\leq v_n$, we can infer that
\begin{align}\label{3.13}
\int_{\mathbb{R}^N}|v_n|^{p_s^*}v_{L,n}^{p(\beta-1)} dx=&\int_{\mathbb{R}^N}|v_n|^{p_s^*-p}|v_n|^{p}v_{L,n}^{p_s^*-p} dx\nonumber\\
=&\int_{\mathbb{R}^N}|v_n|^{p_s^*-p}(v_nv_{L,n}^{\frac{p_s^*-p}{p}})^p dx\nonumber\\
\leq&\int_{v_n<R}R^{p_s^*-p}v_n^{p_s^*}dx+\int_{v_n>R}|v_n|^{p_s^*-p}(v_nv_{L,n}^{\frac{p_s^*-p}{p}})^p dx\nonumber\\
\leq&\int_{v_n<R}R^{p_s^*-p}v_n^{p_s^*}dx\nonumber\\
&+\left(\int_{v_n>R}|v_n|^{p_s^*}dx\right)^{{\frac{p_s^*-p}{p_s^*}}}\left(\int_{v_n>R}(v_nv_{L,n}^{\frac{p_s^*-p}{p}})^{p_s^*}\right)^{\frac{p}{p_s^*}}.
\end{align}
Since $\{v_n\}$ is bounded in $L^{p_s^*}(\mathbb{R}^N)$, we can see that for any $R$ sufficiently large
\begin{align}\label{3.14}
\left(\int_{v_n>R}|v_n|^{p_s^*}dx\right)^{{\frac{p_s^*-p}{p_s^*}}}\leq \frac{1}{2C\beta^p}
\end{align}
putting together (\ref{3.12}), (\ref{3.13}) and (\ref{3.14}), we get
\begin{align*}
\left(\int_{v_n>R}(v_nv_{L,n}^{\frac{p_s^*-p}{p}})^{p_s^*}\right)^{\frac{p}{p_s^*}}\leq C\beta^pR^{p_s^*-p}\int_{\mathbb{R}^N}v_n^{p_s^*}dx<\infty,
\end{align*}
and taking the limiting as $L\to \infty$, we obtain $v_n\in L^{\frac{(p_s^*)^2}{p}}({\mathbb{R}^N})$.\

Now, noting that $0\leq v_{L,n}\leq v_n$ and letting $L\to\infty$ in (\ref{3.12}) we have $|v_n|_{\beta p_s^*}^{\beta p}\leq C\beta^p\int_{\mathbb{R}^N}v_n^{p_s^*+p(\beta-1)}dx$, from which we deduce that
\begin{align*}
\left(\int_{\mathbb{R}^N} v_n^{\beta p_s^*}dx\right)^{\frac{1}{p_s^*(\beta-1)}}\leq C\beta^{\frac{1}{\beta-1}}\left(\int_{\mathbb{R}^N}v_n^{p_s^*+p(\beta-1)}dx\right)^{\frac{1}{p(\beta-1)}}.
\end{align*}
From $m\geq1$, we define $\beta_{m+1}$ inductively so that $p_s^*+p(\beta_{m+1}-1)=p_s^*\beta_m$ and $\beta_1=\frac{p_s^*}{p}$. Then, we have
\begin{align*}
\left(\int_{\mathbb{R}^N} v_n^{\beta_{m+1} p_s^*}dx\right)^{\frac{1}{p_s^*(\beta_{m+1}-1)}}\leq C\beta_{m+1}^{\frac{1}{\beta_{m+1}-1}}\left(\int_{\mathbb{R}^N}v_n^{p_s^*\beta_m}dx\right)^{\frac{1}{p_s^*(\beta_m-1)}}.
\end{align*}
Set $D_m:=\left(\int_{\mathbb{R}^N}v_n^{p_s^*\beta_m}dx\right)^{\frac{1}{p_s^*(\beta_m-1)}}$. Using a standard iteration argument, we can find $C_0>0$ independent of $m$ such that
\begin{align*}
{D_{m + 1}} \le \prod\limits_{k = 1}^m {{{(C{\beta _{k + 1}})}^{\frac{1}{{{\beta _{k + 1}} - 1}}}}} {D_1} \le {C_0}{D_1}.
\end{align*}
Passing to the limit as $m\to\infty$ we get $|v_n|_\infty\leq K$ for all $n\in \mathbb{N}$.\

Now, we note that $v_n$ is solution to
\begin{align*}
(-\Delta)_p^sv_n=\left[g(\varepsilon_n(x+y_n),v_n)-V(\varepsilon_n(x+y_n))v_n^{p-1}\right](a+b[v_n]_{s,p}^{(\theta-1)p})^{-1}=:h_n\ \ \rm{in}\ \ \mathbb{R}^N.
\end{align*}
Moreover, $h_n\in L^\infty(\mathbb{R}^N)$ and $|v_n|_\infty\leq C$ for all $n\in\mathbb{N}$. Indeed, this last inequality is a consequence of the growth assumption on $g$. Then, $|v_n|_\infty\leq C$ and $a\leq a+b[v_n]_{s,p}^{(\theta-1)p}\leq C$ for all $n\in \mathbb{N}$. Then using Corollary $5.5$ in \cite{ASM}, we can deduce that $v_n\in \mathcal{C}^{0,\alpha}(\mathbb{R}^N)$ for some $\alpha>0$. Since $v_n\to v$ in $\mathcal{W}_0$ (see Lemma \ref{conc}), we can infer that
\begin{equation*}
\mathop {\lim }\limits_{|x| \to \infty } {v_n}(x) = 0\ \ \text{uniformly}\ \text{in}\ n\in\mathbb{N}.
\end{equation*}
\end{proof}
Now, we give the proof of Theorem \ref{mainthx}.\

\noindent{\bf Proof of Theorem \ref{mainthx}}.
We begin by proving that there exists $\varepsilon_0>0$ such that for any $\varepsilon\in (0,\varepsilon_0)$ and any mountain pass solution $u_\varepsilon\in \mathcal{W}_\varepsilon$ of (\ref{1.2}), it holds
\begin{align}\label{3.15}
|u_\varepsilon|_{L^\infty(\mathbb{R}^N)\setminus {\Lambda_\varepsilon}}<a.
\end{align}
Assume by contradiction that for some subsequence $\{\varepsilon_n\}$ such that $\varepsilon_n\to 0$, we find $u_n:=u_{\varepsilon_n}\in \mathcal{W}_{\varepsilon_n}$ such that $J_{\varepsilon_n}(u_n)=c_{\varepsilon_n}$, $J'_{\varepsilon_n}(u_n)=0$ and
\begin{align}\label{3.16}
|u_\varepsilon|_{L^\infty({\mathbb{R}^N}\setminus {\Lambda_{\varepsilon_n}})}\geq a.
\end{align}
From Lemma \ref{conc}, there exists $\{y_n\}\subset \mathbb{R}^N$ such that $v_n=u_n(\cdot +y_n)\to v$ in $\mathcal{W}_0$ and $\varepsilon_n y_n\to y_0$ for $y_0\in\Lambda$ such that $V(y_0)=V_0$.\

Now, if we choose $r>0$ such that $B_r(y_0)\subset B_{2r}(y_0)\subset\Lambda$, we can see that $B_{\frac{r}{\varepsilon_n}}\left(\frac{y_0}{\varepsilon_n}\right)\subset\Lambda_{\varepsilon_n}$. Then, for any $y\in B_{\frac{r}{\varepsilon_n}}(y_n)$, it holds $|y-\frac{y_0}{\varepsilon_n}|\leq |y-y_n|+|y_n-\frac{y_0}{\varepsilon_n}|<\frac{1}{\varepsilon_n}(r+o_n(1))<\frac{2r}{\varepsilon_n}$ for $n$ sufficiently large. Hence, for there values of $n$ we have
\begin{align}\label{3.17}
{\mathbb{R}^N}\setminus {\Lambda_{\varepsilon_n}}\subset{\mathbb{R}^N}\setminus {B_{\frac{r}{\varepsilon_n}}(y_n)}
\end{align}
for any $n$ big enough. Using Lemma \ref{estim}, we can see that
\begin{align}\label{3.18}
v_n(x)\to 0\ \ \mbox{as}\ \ |x|\to \infty\ \ \mbox{uniformly}\ \mbox{in}\ n\in\mathbb{N}.
\end{align}
Therefore, there exists $R>0$ such that $v_n(x)<a$ for $|x|\geq R$ $\forall n\in \mathbb{N}$, from which $u_n(x)=v_n(x-y_n)<a$ for $x\in \mathbb{R}^N\setminus B_R(y_n)$ and $\forall n\in\mathbb{N}$. On the other hand, by (\ref{3.17}), there exists $\nu\in\mathbb{N}$ such that for any $n\geq\nu$, it holds ${\mathbb{R}^N}\setminus {\Lambda_{\varepsilon_n}}\subset{\mathbb{R}^N}\setminus {B_{\frac{r}{\varepsilon_n}}(y_n)}\subset\mathbb{R}^N\setminus B_R(y_n)$, which gives $u_n(x)<a$, $\forall x\in{\mathbb{R}^N}\setminus {\Lambda_{\varepsilon_n}}$. This last fact contradicts (\ref{3.16}) and thus (\ref{3.15}) is verified.\

Now, let $u_\varepsilon$ be a nonnegative solution to (\ref{1.2}), since $u_\varepsilon$ satisfies (\ref{3.15}) for any $\varepsilon\in (0,\varepsilon_0)$, it follows from the definition $g$ that $u_\varepsilon$ is a solution to (\ref{1.2}), and $v_\varepsilon=u_\varepsilon(x/\varepsilon)$ is a solution to (\ref{1.1}) for any $\varepsilon\in (0,\varepsilon_0)$.\

Finally, we study the behavior of maximum points of solutions to problem (\ref{1.2}). Taking $\varepsilon_n\to 0^+$ and consider a sequence $\{u_{\varepsilon_n}\}\subset\mathcal{W}_{\varepsilon_n}$ of solution to $(\ref{1.2})$. We first notice that, by $(G_1)$, there exists $\gamma\in (0,a)$ such that
\begin{align}\label{3.19}
g(\varepsilon_n x,t)t=f(t)t+t^{p_s^*}\leq\frac{V_1}{K}t^p\ \ \text{for}\ \text{any}\ \ x\in {\mathbb{R}^N},\ \ 0\leq t\leq\gamma.
\end{align}
Arguing as before, we can find $R>0$
\begin{align}\label{3.20}
|u_{\varepsilon_n}|_{L^\infty({\mathbb{R}^N}\setminus{B_R(y_n)})}<\gamma.
\end{align}
Moreover, up to subsequence, we may assume that
\begin{align}\label{3.21}
|u_{\varepsilon_n}|_{L^\infty({B_R(y_n)})}\geq\gamma.
\end{align}
Indeed, if (\ref{3.21}) does not hold, in view of (\ref{3.20}) we can see that $|u_{\varepsilon_n}|_\infty<\gamma$. Then, from $\left\langle {{J'_\varepsilon }({u_{\varepsilon_n}}),{u_{\varepsilon_n}}} \right\rangle=o_n(1)$ and (\ref{3.19})
\begin{equation*}
\|u_{\varepsilon_n}\|_{\varepsilon_n}^p\leq \|u_{\varepsilon_n}\|_{\varepsilon_n}^p+b[u_{\varepsilon_n}]_{s,p}^{\theta p}=\int_{\mathbb{R}^N}g(\varepsilon_n x,u_{\varepsilon_n})u_{\varepsilon_n} dx\leq \frac{V_1}{K}\int_{\mathbb{R}^N}u_{\varepsilon_n}^p,
\end{equation*}
which yields $\|u_{\varepsilon_n}\|_{\varepsilon_n}=0$, and this is a contradiction. Hence, (\ref{3.21}) holds true. In the light of (\ref{3.20}) and (\ref{3.21}), we can deduce that the maximum points $p_n\in {\mathbb{R}^N}$ of $u_{\varepsilon_n}$ belongs to $B_R(y_n)$. Thus, $p_n=y_n+q_n$ for some $q_n\in B_R$. Consequently, $\eta_{\varepsilon_n}={\varepsilon_n} y_n+\varepsilon_n q_n$ is the maximum point of $v_{\varepsilon_n}=u_{\varepsilon_n}(x/{\varepsilon_n})$, we conclude that the maximum point $\eta_{\varepsilon_n}$ of $v_{\varepsilon_n}$ is given by $\eta_{\varepsilon_n}:=\varepsilon_n y_n+\varepsilon_n q_n$. Since $\{q_n\}\subset B_R$ is bounded and $\varepsilon_n y_n\to y_0$ with $V(y_0)=V_0$, from the continuity of $V$, we can deduce that $\mathop {\lim }\limits_{n \to \infty }V(\eta_{\varepsilon_n})=V(y_0)=V_0$.\

Next, we give a decay estimate for $v_{\varepsilon_n}$. For this purpose, using Lemma $7.1$ in \cite{PQ}, we can find a continuous positive function $w$ and a constant $C>0$ such that for large $|x|>R_0$ it holds that
\begin{align}\label{3.22}
0<w(x)\leq\frac{C}{1+|x|^{N+sp}}
\end{align}
and
\begin{align}\label{3.23}
(-\Delta)_p^s w+\frac{V_1}{2(a+bA_1^{(\theta-1)p})}w^{p-1}\geq 0\ \ \rm{in}\ \ \mathbb{R}^N\setminus B_{R_1},
\end{align}
for some suitable $R_1>0$, and $A_1>0$ is such that $a+b[v_n]_{s,p}^{(\theta-1)p}\leq a+bA_1^{(\theta-1)p}$ for all $n\in \mathbb{N}$. Using $(G_1)$ and (\ref{3.18}), we can find $R_2>0$ sufficiently large such that
\begin{align}\label{3.24}
(-\Delta)_p^s {v_n}+\frac{V_1}{2(a+bA_1^{(\theta-1)p})}{v_n}^{p-1}\leq& (-\Delta)_p^s {v_n}+\frac{V_1}{2(a+b[v_n]_{s,p}^{(\theta-1)p})}{v_n}^{p-1}\nonumber\\
=&\frac{1}{2(a+b[v_n]_{s,p}^{(\theta-1)p})}\left[g(\varepsilon_n x+\varepsilon_n y_n,v_n)-(V(\varepsilon_n x+\varepsilon_n y_n)-\frac{V_1}{2}){v_n}^{p-1}\right]\nonumber\\
\leq&\frac{1}{2(a+b[v_n]_{s,p}^{(\theta-1)p})}\left[g(\varepsilon_n x+\varepsilon_n y_n,v_n)--\frac{V_1}{2}{v_n}^{p-1}\right]\nonumber\\
\leq&0,\ \ \mbox{in}\ \ \mathbb{R}^N\setminus {B_{R_2}}.
\end{align}
Thanks to the continuity of $v_n$ and $w$. there exists $C_1>0$ such that
\begin{equation*}
\psi_n:=v_n-C_1w\leq0 \ \ \mbox{for}\ \ |x|=R_3,
\end{equation*}
where $R_3:=\max\{R_1, R_2\}$ . \

Taking $\phi_n=\max\{\psi_n, 0\}\in W^{s,p}(B_R^c)$ as a test function in (\ref{3.24}) and using (\ref{3.23}) with $\tilde w=C_1w$, we can deduce that
\begin{align}\label{3.25}
0\geq&\iint_{\mathbb{R}^{2N}}\frac{{|{v_n}(x) - {v_n}(y){|^{p - 2}}({v_n}(x) - {v_n}(y))({\phi _n}(x) - {\phi _n}(y))}}{{|x - y{|^{N + sp}}}} dxdy\nonumber\\
&+\frac{V_1}{2(a+bA_1^{(\theta-1)p})}\int_{\mathbb{R}^N} v_n^{p-1}\phi_n dx\nonumber\\
\geq& \iint_{\mathbb{R}^{2N}}\frac{\mathcal{G}_n(x,y)}{{|x - y{|^{N + sp}}}}({\phi _n}(x) - {\phi _n}(y))dxdy+\frac{V_1}{2(a+bA_1^{(\theta-1)p})}\int_{\mathbb{R}^N}[v_n^{p-1}-{{\tilde w}^{p - 1}}]\phi dx,
\end{align}
where ${\mathcal{G}_n(x,y)}:={|{v_n}(x) - {v_n}(y){|^{p - 2}}({v_n}(x) - {\tilde w}(y))}-{|{\tilde w}(x) - {\tilde w}(y){|^{p - 2}}({\tilde w}(x) - {\tilde w}(y))}$. Therefore, if we prove that
\begin{align}\label{3.26}
\iint_{\mathbb{R}^{2N}}\frac{\mathcal{G}_n(x,y)}{{|x - y{|^{N + sp}}}}({\phi _n}(x) - {\phi _n}(y))dxdy\geq0,
\end{align}
it follows from (\ref{3.25}) that
\begin{equation*}
0\geq\frac{V_1}{2(a+bA_1^{(\theta-1)p})}\int_{v_n\geq {\tilde w}}[v_n^{p-1}-{\tilde w}^{p-1}](v_n-{\tilde w})dx\geq0,
\end{equation*}
which implies that
\begin{equation*}
\{x\in \mathbb{R}^N:|x|\geq R_3\ \ and\ \ v_n\geq {\tilde w}\}=\emptyset.
\end{equation*}
To obtain our purpose, we first note that for all $s,d\in\mathbb{R}$ it holds
\begin{equation*}
|d|^{p-2}d-|c|^{p-2}c=(p-1)(d-c)\int_0^1|c+t(d-c)|^{p-2} dt.
\end{equation*}
Taking $d=v_n(x)-v_n(y)$ and $c={\tilde w}(x)-{\tilde w}(y)$ we can see that
\begin{equation*}
|d|^{p-2}d-|c|^{p-2}c=(p-1)(d-c)I(x,y),
\end{equation*}
where $I(x,y)\geq 0$ stands for the integral. Now, recalling that
\begin{equation*}
(x-y)(x^+-y^+)\geq |x^+-y^+|^2\ \ \mbox{for}\ \mbox{all}\ \ x,y\in \mathbb{R},
\end{equation*}
one has
\begin{align*}
(d-c)(\phi_n(x)-\phi_n(y))=&[(v_n-{\tilde w})(x)-(v_n-{\tilde w})(y)][(v_n-{\tilde w})^+(x)-(v_n-{\tilde w})^+(y)]\\
\geq& |(v_n-{\tilde w})^+(x)-(v_n-{\tilde w})^+(y)|^2,
\end{align*}
which gives $(|d|^{p-2}d-|c|^{p-2}c)(\phi_n(x)-\phi_n(y))\geq 0$, that is (\ref{3.26}) holds.\

Therefore, $\psi_n \leq0$ in $B_{R_3}^c$, which implies that $v_n\leq C_1w$ in $B_{R_3}^c$, that is $v_n(x)\leq C\frac{1}{1+|x|^{N+sp}}$ in $B_{R_3}^c$. Therefore, we have
\begin{align*}
v_{\varepsilon_n}(x)=&{u_{\varepsilon_n}}\left(\frac{x}{\varepsilon_n}\right)=v_n\left(\frac{x}{\varepsilon_n}-y_n\right)\leq C_1w\left(\frac{x}{\varepsilon_n}-y_n\right)\\
\leq& \frac{C}{1+|\frac{x}{\varepsilon_n}-y_n|^{N+sp}}\\
=&\frac{C{\varepsilon_n}^{N+sp}}{{\varepsilon_n}^{N+sp}+|x-{\varepsilon_n}y_n|^{N+sp}}\\
\leq&\frac{C{\varepsilon_n}^{N+sp}}{{\varepsilon_n}^{N+sp}+|x-\eta_{\varepsilon_n}|^{N+sp}}.
\end{align*}
The proof is complete.
\qed

\bigskip

\noindent{\bfseries Acknowledgements:}

The research bas been supported by National Natural Science Foundation of China 11971392,  Natural Science Foundation of Chongqing, China cstc2021ycjh-bgzxm0115, Fundamental Research Funds for the Central Universities XDJK2020B047, and the Chongqing Graduate Student Research Innovation Project CYS21098.


\begin{thebibliography}{9}
\bibitem{AOO}
C. O. Alves and O. H. Miyagaki, {\em Existence and concentration of solution for a class of fractional elliptic equation in RN via penalization method}, Calc. Var. Partial Differ. Equ. 55 (2016), 47. 19 pp.



\bibitem{VA1}
V. Ambrosio, {\em Multiplicity of positive solutions for a class of fractional Schr\"{o}dinger equations via penalization method},
Ann. Mat. Pur. Appl. 196(4) (2017), 2043-2062.



\bibitem{VA2}
V. Ambrosio, {\em Concentration phenomena for critical fractional Schr\"{o}dinger systems}, Commun.Pure. Appl. Anal. 17(5) (2018), 2085-2123.



\bibitem{VA3}
V. Ambrosio, {\em Concentrating solutions for a class of nonlinear fractional Schr\"{o}dinger equations in $\mathbb{R}^N$},
Rev. Mat. Iberoam. 35 (2019), 1367-1414.


\bibitem{VA4}
V. Ambrosio and T. Isernia, {\em Concentrating of positive solutions for a class of fractional p-Kirchhoff type equations},
Proc. Roy. Soc. Edinburgh Sect. A,
151 (2021), 601-651.


\bibitem{AI}
V. Ambrosio and T. Isernia, {\em Multiplicity and concentration results for some nonlinear Schr\"{o}dinger equations with the fractional p-Laplacian}, Discrete Contin. Dyn. Syst. 38 (2018), 5835-5881.




\bibitem{VA}
V. Ambrosio, {\em Fractional  $p\&q$ Laplacian problems in $\mathbb{R}^N$ with critical growth},
Z. Anal. Anwed.,
39 (2020), 289-314.

\bibitem{AFI}
V. Ambrosio, G. M. Figueirdo and T. Isernia, {\em Existence and concentration of positive solutions for p-Laplacian Schr\"{o}dinger equations},
Ann. Mat. Pur. Appl.,
199 (2020), 317-344.

\bibitem{BC}
V. Benci and G. Cerami, {\em Multiple positive solutions of some elliptic problems via the Morse theory and the domain topology}, Calc. Var. Partial Differ. Equ. 2 (1994), 29-48.


\bibitem{SB}
S. Bernstein, {\em Sur une classe d'\'{e}quations fonctionnelles aux d\'{e}riv\'{e}es partielles}, Bull. Acad. Sci. URSS. S\'{e}r. Math. [Izvestia Akad. Nauk SSSR] 4 (1940), 17-26.





\bibitem{LS}
L. Brasco, S. Mosconi and M. Squassina, {\em Optimal decay of extremals for the fractional Sobolev inequality}, Calc. Var. Partial Differ. Equ. 55 (2016), 1-32.



\bibitem{KCC}
K. C. Chang, {\em Infinite dimensional morse theory and multiple solution problems}, Birkh\"{a}user, Boston, 1993.




\bibitem{classif}
W. Chen, C.\ Li, B.\ Ou, {\em Classification of solutions for an integral equation},
Comm. Pure Appl. Math., 59(2006), 330-343.


\bibitem{CG}
W. Chen and Y. Gui, {\em Multiple solutions for a fractional $p$-Kirchhoff problem with Hardy
nonlinearity}, Nonlinear Anal. 188(2019), 316-338.



\bibitem{CSS}
W. Chen, S. Mosconi and M. Squassina, {\em Nonlocal problems with critical Hardy nonlinearity}, J. Funct. Anal., 275(2018), 3065-3114.




\bibitem{JMJ}
J. D\'{a}vila, M. del Pino and J. Wei, {\em Concentrating standing waves for the fractional nonlinear Schr\"{o}dinger equation}, J. Differ. Equ. 256 (2014), 858-892.

\bibitem{PQ}
L. M. del Pezzo and A. Quaas, {\em Spectrum of the fractional $p$-Laplacian in $\mathbb{R}^N$ and decay estimate for positive solutions of a Schr\"{o}dinger equation}, Nonlinear Anal. 193(2020), 111479.

\bibitem{PMQ}
L. M. del Pezzo and A. Quass, {\em A Hopf's lemma and a strong minimum principle for the fractional $p$-Laplacian}, J. Differ. Equ. 263 (2017), 765-778.


\bibitem{PF}
M. del Pino and P. Felmer, {\em Local mountain passes for semilinear elliptic problems in unbounded domains}, Calc. Var. Partial Differ. Equ. 4 (1996), 121-137.

\bibitem{NEGE}
E. Di Nezza, G. Palatucci and E. Valdinoci, {\em Hitchhiker's guide to the fractional Sobolev spaces}, Bull. Sci. Math. 136 (2012), 521-573.


\bibitem{PAJ}
P. Felmer, A. Quaas and J. Tan,  {\em Positive solutions of the nonlinear Schr\"{o}dinger equation with the fractional Laplacian},
Proc. Roy. Soc. Edinburgh Sect. A, 142 (2012), 1237-1262.

\bibitem{GJ}
G. M. Figueiredo and J. R. Santos, {\em Multiplicity and concentration behavior of positive solutions for a Schr\"{o}dinger-Kirchhoff type problem via penalization method},   ESAIM Control Optim. Calc. Var. 20 (2014), 389-415.






\bibitem{FP}
A. Fiscella and P. Pucci, {\em Kirchhoff-Hardy fractional problems with lack of compactness}, Adv. Nonlinear Stud. 17 (2017), 429-456.



\bibitem{HLP}
Y. He, G. Li and S. Peng, {\em Concentrating bound states for Kirchhoff type problems in $\mathbb{R}^3$ involving critical Sobolev exponents}, Adv. Nonlinear Stud. 14 (2014), 483-510.


\bibitem{HZ}
X. He and W. Zou, {\em Existence and concentration behavior of positive solutions for a Kirchhoff equation in $\mathbb{R}^3$}, J. Differ. Equ. 252(2012), 1813-1834.


\bibitem{ASM}
A. Iannizzotto, S. Mosconi and M. Squassina, {\em Global H\"{o}lder regularity for the fractional p-Laplacian}, Rev. Mat. Iberoam. 32 (2016), 1353-1392


\bibitem{T}
T. Isernia, {\em Positive solution for nonhomogeneous sublinear fractional equations in $\mathbb{R}^N$}, Complex Var. Elliptic Equ. 63(5) (2018), 689-714.




\bibitem{K}
G. Kirchhoff, {\em  Mechanik}. (Leipzig: Teubner, 1883).

\bibitem{L}
N. Laskin, {\em Fractional quantum mechanics and L\'{e}vy path integrals}, Phys.Lett. A, 268(2000), 298-305.



\bibitem{Lion1}
P. L. Lions, {\em The concentration-compactness principle in the calculus of variations. The limit case. Part 1}, Rev. Mat. Iberoam. 1(1)(1985), 145-201.




\bibitem{SKS}
S. Mosconi, K. Perera, M. Squassina and Y. Yang, {\em The Brezis-Nirenberg problem for the fractional $p$-Laplacian}, Calc. Var. Partial Differ. Equ. 55 (2016), 25.


\bibitem{MRZ}
X. Mingqi, V. R\u{a}dulescu and B. Zhang, {\em Multiplicity of solutions for a class of quasilinear Kirchhoff system involving the fractional p-Laplacian}, Nonlinearity 29 (2016), 3186-3205.




\bibitem{Moser}
J. Moser, {\em A new proof of De Giorgi's theorem concerning the regularity problem for elliptic differential equations}, Commun. Pure Appl. Math. 13 (1960), 457-468






\bibitem{SIP}
S. I. Poho\u{z}aev, {\em A certain class of quasilinear hyperbolic equations}, Mat. Sb. 96 (1975), 152-166.






\bibitem{PXZ}
P. Pucci, M. Xiang and B. Zhang, {\em Multiple solutions for nonhomogeneous Schr\"{o}dinger-Kirchhoff type equations involving the fractional p-Laplacion in $\mathbb{R}^N$}, Calc. Var. Paticial Differ. Equ. 54 (2015), 2785-2806.





\bibitem{PR}
P. Rabinowitz, {\em On a class of nonlinear Schr\"{o}dinger equations},  Z. Angew. Math. Phys. 43 (1992), 270-291.


\bibitem{Si}
J. Simon, {\em Re\'{g}ularit\'{e} de la solution d'une \'{e}quation non lin\'{e}aire dans ${\mathbb{R}}^N$}, Lecture Notes in Math. 665 (1978), 205-227.




\bibitem{S}
S. Secchi,  {\em Ground state solutions for nonlinear fractional Schr\"{o}dinger equations in $\mathbb{R}^N$}, J. Math. Phys. 54 (2013), 031501.



\bibitem{WTXZ}
J. Wang, L. Tian, J. Xu and F. Zhang, {\em Multiplicity and concentration of positive solutions for a Kirchhoff type problem with critical growth},  J. Differ. Equ. 253 (2012), 2314-2351.

\bibitem{W}
M. Willem, Minimax Theorems, vol. 24, Birkh\"{a}user Boston Inc, Boston, MA, 1996.


\end{thebibliography}
\enddocument